\pdfoutput=1
\documentclass{article} 
\usepackage{amsmath}
\usepackage{amssymb}
\usepackage{graphicx}
\usepackage{epstopdf}
\usepackage{subfig}
\usepackage{inputenc}
\usepackage{lineno}
\usepackage{geometry}
\usepackage{authblk}
\usepackage{url}
\usepackage{soul}
\usepackage{amsmath, amsthm, amssymb}
\usepackage{graphicx}
\usepackage{subfig}
\usepackage{float}
\newtheorem{prop}{Proposition}
\usepackage{algorithm,algcompatible,amsmath}
\usepackage{multicol}
\usepackage{multirow}
\usepackage{bbm}
\theoremstyle{plain}

\theoremstyle{definition}

\usepackage{vwcol}

\usepackage{xcolor}

\usepackage[overload]{empheq}
\usepackage{algpseudocode}
\usepackage{algorithmicx}
\usepackage{amsmath}
\usepackage{enumitem}
\algrenewcommand\textproc{}
\DeclareMathOperator*{\argmax}{arg\,max}

\newlist{steps}{enumerate}{1}
\setlist[steps, 1]{label = Step \arabic*:}
\usepackage{booktabs}

\providecommand{\keywords}[1]{\textbf{\textit{Keywords---}} #1}
\def\lc{\left\lceil}   
\def\rc{\right\rceil}

\usepackage{setspace}
\newcommand{\Tau}{\mathrm{T}}
\begin{document}


\title{Demand-Adaptive Route Planning and Scheduling for Urban Hub-based High-Capacity Mobility-on-Demand Services\\
\large \textit{Submit to ISTTT 24}}

\author[1]{Xinwu Qian}
\affil[1]{Lyles School of Civil Engineering, Purdue University\\qian39@purdue.edu}

\author[2]{Jiawei Xue}
\affil[2]{Lyles School of Civil Engineering, Purdue University\\xue120@purdue.edu}

\author[3]{Satish V. Ukkusuri}
\affil[3]{Lyles School of Civil Engineering, Purdue University\\sukkusur@purdue.edu}

\date{}
\maketitle

\begin{abstract}
In this study, we propose a three-stage framework for the planning and scheduling of high-capacity mobility-on-demand services (e.g., microtransit and flexible transit) at urban activity hubs. The proposed framework consists of (1) the route generation step to and from the activity hub with connectivity to existing transit systems, and (2) the robust route scheduling step which determines the vehicle assignment and route headway under demand uncertainty. Efficient exact and heuristic algorithms are developed for identifying the minimum number of routes that maximize passenger coverage, and a matching scheme is proposed to combine routes to and from the hub into roundtrips optimally. With the generated routes, the robust route scheduling problem is formulated as a two-stage robust optimization problem. Model reformulations are introduced to solve the robust optimization problem into the global optimum. In this regard, the proposed framework presents both algorithmic and analytic solutions for developing the hub-based transit services in response to the varying passenger demand over a short-time period. To validate the effectiveness of the proposed framework, comprehensive numerical experiments are conducted for planning the HHMoD services at the JFK airport in New York City (NYC). The results show the superior performance of the proposed route generation algorithm to maximize the citywide coverage more efficiently. The results also demonstrate the cost-effectiveness of the robust route schedules under normal demand conditions and against worst-case-oriented realizations of passenger demand. 
\end{abstract}

\section{Introduction}
High-capacity public transit such as bus and metro serves as an affordable mobility solution to urban commuters. When properly planned and operated, public transit may significantly reduce commuters' dependence on private vehicles and play a vital role in a sustainable mobility system in dense populated urban areas~\cite{zhang2010can}. Unfortunately, the configuration and infrastructure of existing fixed-route public transit systems may no longer be attractive to urban travelers due to their fast-changing mobility needs. One notable evidence is the recent rise of the ride-hailing industry and subsequent loss of public transportation ridership. For instance, in New York City (NYC), the bus ridership declined by 1.3\%, 5.1\%, and 5.8\% from years 2016 to 2018~\cite{nycMTA2018} while the number of for-hire vehicles (FHV) trips had increased by 300\% during the time~\cite{nycfhv2019}. This leads to more congestion, emissions, and energy consumption~\cite{qian2020impact} in cities, and poses a significant need for a better balance and connectivity between emerging mobility options and public transportation. 

One challenge associated with the planning of the transit system is the trade-off between accessibility and service coverage. It is often impractical for high-capacity urban transit such as fixed-route buses to provide door-to-door service to meet the varying mobility needs of the general public. Nevertheless, with flexible transit of adaptive service routes and schedules, it is possible to offer a more effective transit service for serving passengers sharing similar origins and destinations~\cite{malucelli1999demand}. One appealing and applicable scenario is the public transit rooted in major urban activity hubs such as shopping malls, financial districts, railway stations, and airports. These locations have high passenger volume, and it is likely to find travelers to and from these hubs of similar travel routes. We show in Figure~\ref{fig:cum_arrival_departure} the empirical evidence that supports the applicability of hub-based transit service, where the cumulative pickups and drop-offs of both taxis and FHVs are plotted at the taxi-zone level. In particular, we can verify that around 20\% of the taxi zones account for over 60\% of total pickups and drop-offs for both taxis and FHV. Moreover, we observe from the data that over 25\% of yellow taxi trips and 15\% of FHV trips have either the pickup or drop-off at one of the airports in NYC during peak hours. These observations reveal huge opportunities for travel mileage reduction if individual trips can be optimally consolidated, and motivate us to investigate the design of the hub-based high-capacity mobility-on-demand (HHMoD) service as a demand-adaptive transit system at urban activity hubs. The HHMoD is a special case of the ridesharing scheme, and the main idea is to satisfy passenger demand at activity hubs with high capacity vehicles that use optimized routes with minimum deviation from the shortest travel routes. But different from the ad-hoc ridesharing with fewer passengers, the high capacity feature and the higher and more stable passenger demand at activity hubs allow the routes and schedules of the HHMoD to be planned proactively. This avoids the unnecessary waiting for the realization of requests from many passengers and is also more attractive to potential passengers as they will be informed of the scheduled services ahead of time. And the routes and schedules of the HHMoD are adaptive since they can be reoptimized from time to time (e.g., every 30 minutes or one hour) based on the predicted future demand to ensure the planned services are flexible and aligned with the time-varying mobility needs. 
\begin{figure}[ht!]
    \centering
    \subfloat[Drop-offs]{\includegraphics[width=0.3\linewidth]{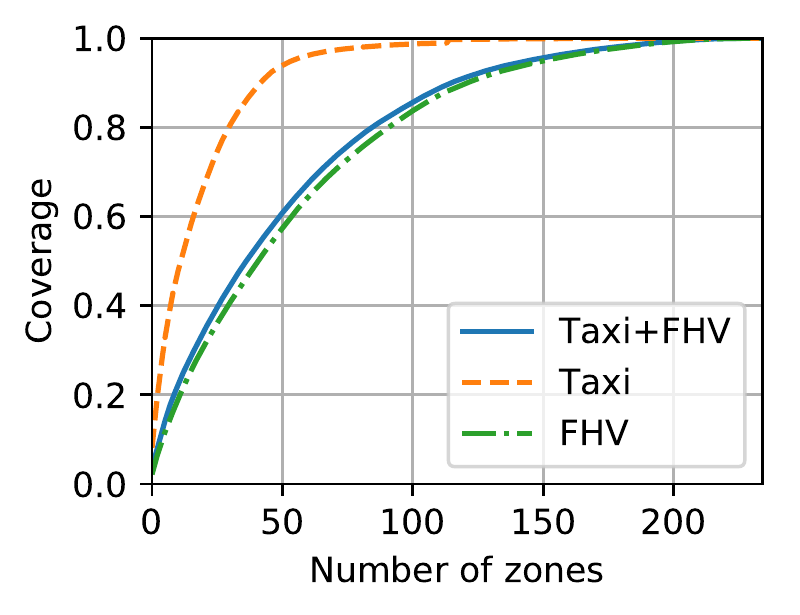}}
    \subfloat[Pick-ups]{\includegraphics[width=0.3\linewidth]{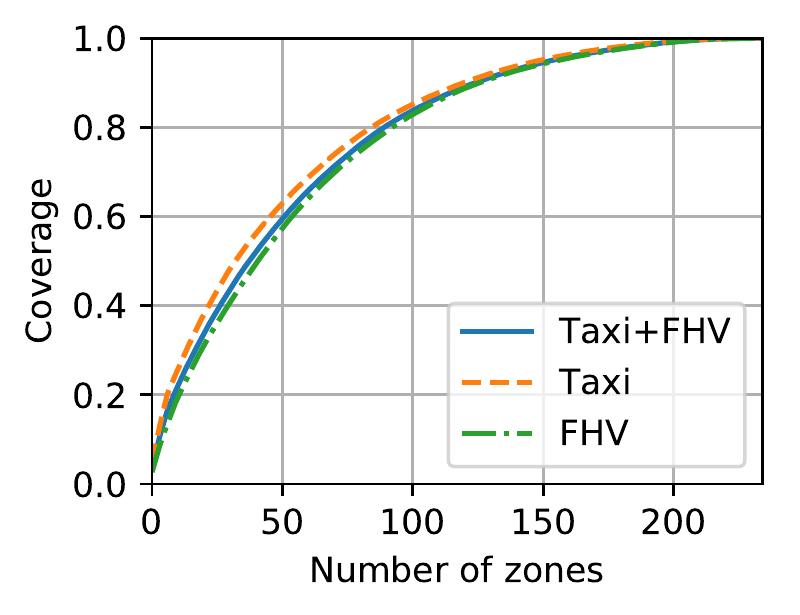}}
    \caption{Zonal-level taxi and FHV pickups and drop-offs in NYC. There are 263 taxi zones.}
    \label{fig:cum_arrival_departure}
\end{figure}

To promote the idea of the HHMoD, this study concerns the modeling framework for the route planning and scheduling of the HHMoD over a short time period, which can be regarded as a novel variant of the classical transit route network design problem (TRNDP). There exists broad literature on TRNDP and readers may refer to~\cite{transitglobal,trn2009,trn2013} for comprehensive reviews of related studies. In general, the TRNDP is intrinsically difficult and the analytical framework cannot be applied to solve real-world problems. In this regard, heuristic approaches constitute the mainstream for solving TRNDP, which compromises sub-problems of the design of operation routes and the scheduling of transit service~\cite{trn2009}. There are two main directions to tackle the sub-problems. The first direction divides the process of transit planning into route generation and headway optimization. Most of the TRNDP studies followed this direction and various methods for route generation were proposed. For instance, Lampkin and Saalmans~\cite{lampkin1967} proposed the route generation method that starts with random control points and consecutively adds stops between these points. Silman et al.~\cite{Silman1974} divided the city into zones and used a heuristic algorithm for generating routes that minimize walking time. Recently, Cipriano et al.~\cite{cipriani2012} used Genetic Algorithms (GA) to generate candidate route sets, which were evolved according to the fitness measure of travel time. Nikolic and Teodorovic~\cite{bee2013} developed the route generation process based on Bee Colony Optimization, where random routes got improved by modeling each route as a bee and modifying its route with a pheromone objective function. The second direction follows an iterative route generation and improvement approach where route and headway are jointly determined in an iterative process. Studies that follow the second direction include Drezner et al.~\cite{drezner2002}, Fan and Wei~\cite{fan2004optimal}, Chakroborty and Partha~\cite{genetic2003}, where GA was widely adopted in these studies to create routes and evolve the routes over iterations, and travel links or control points were chosen as genes which may change during the search of an optimal solution. 
    

It can be seen from the literature that, due to the difficulties in solving TRNDP exactly, heuristic and metaheuristic approaches are heavily used to obtain the candidate route set and determine the frequency setting. Nevertheless, there is no guarantee on the solution quality with these approaches and the complicated parameter tuning is often required in order to obtain a satisfactory solution. Moreover, even for a fine-tuned heuristic method, the performance may vary significantly with respect to different passenger demand realizations, which creates a critical barrier for using such a framework to solve the time-varying flexible route transit design problem. We also note that there are several recent studies that investigated the design of flexible route transit and the demand-adaptive transit, such as the simulation model for demand responsive transit design~\cite{quadrifoglio2008simulation}, the demand-responsive system with compulsory and optional stops~\cite{crainic2012designing}, the flexible transit for low demand areas in grid networks~\cite{nourbakhsh2012structured}, the semi-flexible transit systems~\cite{errico2013survey}, the feeder transit with fixed and flexible routes~\cite{li2010feeder}, the point-to-point high coverage transit system~\cite{cortes2002design}, and the paired-line hybrid transit with radial route structure~\cite{chen2017analysis,luo2020paired} and the flexible line length bus system~\cite{pei2019real}. There exist three major gaps based on the results from the previous studies. First, as the flexible route service was primarily proposed for sparsely populated areas, the effectiveness of the existing framework remains unknown if implemented in populated urban areas. Second, existing studies rely on the simulation framework~\cite{quadrifoglio2008simulation,cortes2002design}, heuristics~\cite{li2010feeder} or metaheuristics~\cite{chen2017analysis} for obtaining the optimal design due to the high computational cost. Finally,
the demand-responsive mobility services were mainly investigated under deterministic or revealed demand. But an effective demand-responsive system needs to account for future passenger realizations under uncertainty. To the best of our knowledge, an analytical framework for the route generation and robust route scheduling of urban mobility-on-demand services has yet been established.  

In light of the fundamental challenges and the practical usefulness of the HHMoD for serving the mobility needs in populated urban areas, this study establishes the framework for operating HHMoD that is tailored to the spatiotemporally varying mobility needs of urban travelers based on the real-time trip and traffic data. The primary goal of the planned HHMoD is to find the set of routes that maximize passenger coverage and identify the optimal schedules of the routes that minimize the worst-case oriented operation cost under demand uncertainties. Specifically, we take advantage of the hub-based structure and propose both exact and heuristic algorithms that generate the round-trip HHMoD routes with maximum passenger coverage. The solution algorithms are further extended to incorporate connections to existing public transportation systems. Given the candidate routes, a two-stage robust optimization model is then developed to schedule the vehicle assignment and route headway optimally. The computational performances and the effectiveness of the proposed HHMoD solution are demonstrated with large-scale numerical experiments at the JFK airport in NYC with 200 candidate stops and connectivity to the NYC subway system. 


The rest of the study is organized as follows. The second section gives an overview of the research framework for developing the HHMoD system. The third section discusses the route generation problems and
proposes exact and heuristic algorithms for obtaining the candidate route
set. A heuristic shortest-path based route generation mechanism in the literature is also introduced as the benchmark algorithm. Afterward, we present the route combination model for joining planned routes to and from the hub. The fourth section develops the two-stage robust optimization model in light of the worst passenger realization for the optimal scheduling given the generated route set. The fifth presents comprehensive numerical experiments on planning the HHMoD service at JFK airport in NYC. Finally, we conclude our study with major findings and future directions in the last section.
\section{Problem overview}
\begin{figure}[h!]
    \centering
    \includegraphics[width=0.8\linewidth]{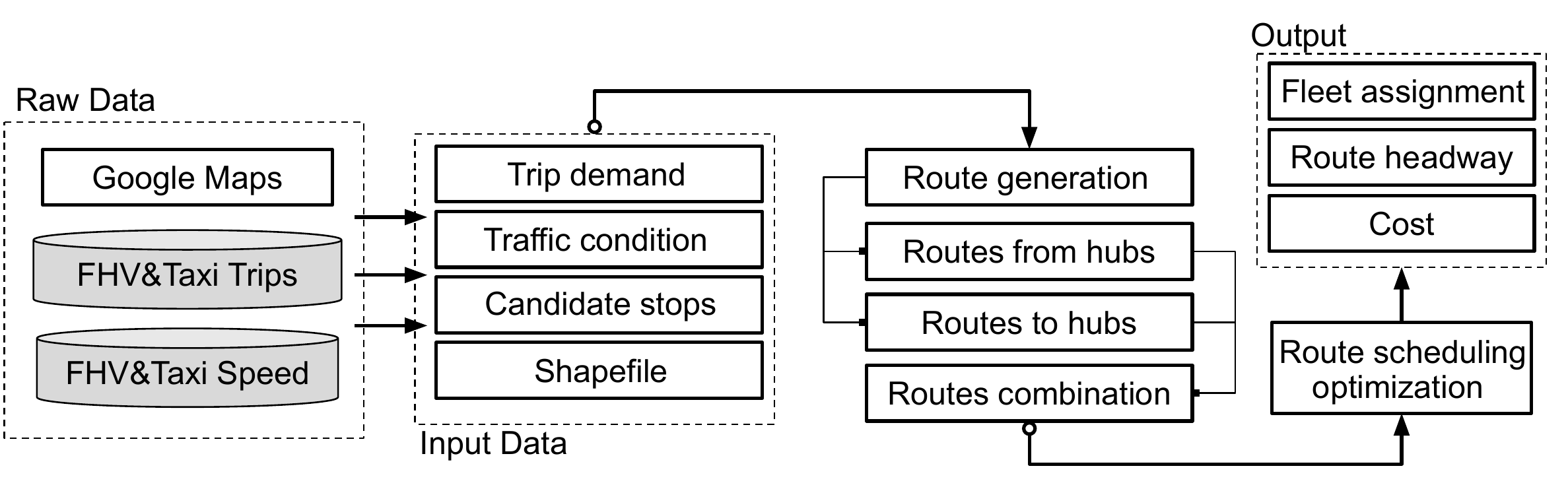}
    \caption{HHMoD Planning framework}
    \label{fig:framework}
\end{figure}
In this study, we follow the transit planning framework as summarized in~\cite{trn2009} by first generating the set of candidate routes and then developing the optimal service frequency under demand uncertainty. We propose a three-stage framework for developing the HHMoD system, as shown in Figure~\ref{fig:framework}. Our framework is different from the route generation step in existing studies - we propose an innovative route generation algorithm that is objective-driven (maximizing passenger coverage). One may choose to find a predefined number of routes with maximum passenger coverage or find the minimum number of routes that cover all passenger demand. Besides, the route generation algorithm offers additional flexibility for connecting the HHMoD system to existing transportation facilities. The overall framework takes the set of candidate stop locations, the trip demand level, and road traffic conditions as the inputs. It outputs the set of operating routes, the assigned number of vehicles to each route, and each route's corresponding operating frequency (headway). Here we briefly summarize the three major steps for the development of the HHMoD system:
\begin{enumerate}
    \item In the first stage, we develop the route generation algorithm which builds the set of $K$ \textbf{feasible} candidate routes that maximize the coverage of potential passengers, where $K$ is the user-specified number. A route is considered as feasible if the trip time of all passengers is no larger than $\lambda$ times the minimum trip time by taking alternative travel modes. Given the candidate stops, trip demand and road traffic condition, the route generation algorithm will first create two sets of feasible routes for incoming and outgoing trips at activity hubs respectively.
    \item The second step then identifies the optimal combination of the two sets of routes to and from the hub. The route combination step focuses on balancing the passenger demand between two directions and is subject to the vacant trip distance constraint. The constraint helps to avoid excessive empty miles for connecting two distant routes and will force the vehicle to return to the hub if there are no matching routes in the opposite direction within the distance constraint.
    \item The integrated round trip routes will then serve as the input for the third stage problem, which is to optimize the vehicle assignment and route service frequency under demand uncertainty. For the route scheduling optimization problem, our objective is to minimize the weighted combination of the total operation cost, passenger waiting time together with the penalty for unsatisfied passenger demand. A two-stage robust optimization problem is developed to identify the worst-case oriented optimal fleet and headway setting. 
\end{enumerate}

As the problem requires the information on candidate stop locations and the associated demand level, such knowledge can be directly mined from large-scale GPS data~\cite{chen2013b}. And the real-time travel time information can also be obtained from various sources such as Google Maps. With the availability of this information, the proposed method can be used to deploy the HHMoD service over a short-time period (e.g., 30 minutes or an hour interval) based on predicted demand or the distribution of future passenger demand from historical observations. This offers great flexibility to align the available fleet resources with the passengers' time-varying mobility needs, and the framework may, therefore, advance the planning and operation of urban MoD with high capacity vehicles. We next present the details of the three-stage framework. 

\section{Route generation}
\subsection{Problem description}
As the first step of the HHMoD development, the route generation problem focuses on identifying the set of $K$ candidate routes that yield the highest possible passenger coverage. We term this target set of $K$ routes as the K maximum coverage routes (K-MCR). We consider the study area consists of $|N|$ candidate stops and we denote $G(N,E,Q,W,\lambda)$ as the network representation of the study area, with $N$ being the set of stops, $e_{i,j}\in E$ denoting the shortest travel path between stop $i$ and stop $j$, $q_i\in Q$ being the potential passenger demand (to hub or from hub) at stop $i$, and $w_{i,j}\in W$ being the weight (shortest travel time) on edge $e_{i,j}$. We make the following reasonable assumptions to conceptualize the routing generation problem:
\begin{enumerate}
    \item We assume that the HHMoD is a station-based system where the candidate stops and the associated demand level at each stop is given.
    \item The travel time between two stops are known (e.g., obtained from historical or real-time traffic data). 
    \item We consider that passengers arriving at station $i$ are covered by the HHMoD if there exists a route that stops at $i$\label{assumption: disjoint}.  
    \item The HHMoD service will use the shortest travel time path when travelling between two stops.
    \item The shortest travel time on different segments between stations should satisfy the triangular inequality constraint\label{assumption:tria}. 
    \item Passengers using the HHMoD are time sensitive. 
\end{enumerate}
The first four assumptions set the scope of the route generation problem with candidate stop locations, known passenger demand and travel time information. The fifth assumption imposes a practical constraint on travel time, $t_{i,j}+t_{j,k}>t_{i,k}$, so that the route travel time should strictly increase with additional intermediate stops. And the sixth assumption suggests that passengers may use alternate travel mode with trip time $t^m$ (such as subway and taxis) if the travel time of the proposed HHMoD service exceeds a certain threshold. This motivates us to take the route travel time into consideration while deriving the MCR, where we can translate the time constraints for passengers who arrive at stop $i$ using route $k$ as:
\begin{equation}
    t_{i}^k\leq \lambda \min_m  t_{i}^m 
    \label{eq:time_constraint}
\end{equation}
so that the planned route $k\in \mathcal{R}_K$ for passengers at stop $i$ should result in the travel time no greater than $\lambda$ times that of the alternative modes of lowest travel time. As an example, with $\lambda=1.5$, passengers will not use the planned HHMoD service if the travel time is 50\% higher than the travel time of hailing a taxi. In this regard, $\lambda$ is a hyper-parameter of the route generation model and can be adjusted to account for the willingness to wait for potential passengers.  

Based on the above constraints, the K-MCR generation problem can be mathematically formulated as:
\begin{equation}
\begin{aligned}
    \text{maximize} & \sum_{i=1}^{|N|} \delta_i^k s_kq_i\\
    \text{s.t.}\quad\quad  t_i^k&\leq \lambda \text{min}  t_{i}^m\, ,\forall i,k\\
    \sum_{k} s_k&\leq K, s_k\in\{0,1\}
\end{aligned}
\label{eq:opt_route}
\end{equation}
where $\delta_i^k$ and $s_k$ are both indicator variables. $s_k=1$ if route $k$ is selected for the K-MCR and 0 otherwise. $\delta_i^k$ takes the value of 1 if stop $i$ is served by route $k$.

Problem~\ref{eq:opt_route} is NP-hard. A straightforward proof is that we can reduce the longest path problem, a known NP-complete problem, to Problem~\ref{eq:opt_route} under $K= 1$ by setting $\lambda$ to be an arbitrary large value. With such a $\lambda$, all travel time constraints are directly satisfied and $K=1$ is a special case of the original problem. The problem becomes more difficult with active travel time constraints and $K>1$. And it is impractical to solve it as an optimization problem with commercial solvers as the explicit formulation may require the enumeration of all candidate routes. In the following sections, we establish exact and heuristic solution algorithms to solve the K-MCR generation problem. 

\subsection{Exact algorithm}
While Problem~\ref{eq:opt_route} represents a difficult combinatorial problem, we can still solve it efficiently by utilizing the topological properties of the network and specific hub-based nature of the problem.

First, denote $\mathcal{R}_K$ as an optimal solution of problem~\ref{eq:opt_route} and the route $r_i\in \mathcal{R}_K$ as the ordered set of stops it traverses through, e.g. $r_i=\{n_1,n_2,n_3,...n_{L_i}\}$ and $n_1,n_2,...n_{L_i}\in N$. Based on Assumption~\ref{assumption: disjoint}, $\mathcal{R}_K$ can be reduced to $\mathcal{R}_K'$ where for any $r_i,r_j\in \mathcal{R}_K'$, we have $r_i\bigcap r_j=\emptyset$. And we term $\mathcal{R}_K'$ as the mutually disjoint K-MCR. To see this, let $r_i\bigcap r_j=S \neq \emptyset$ and we may remove the set of nodes in $S$ from either $r_i$ or $r_j$. This step will not affect the coverage of passengers as the removed stops from one route is still served by the other route. In addition, the route after the stop removal is still feasible since the route travel time strictly decreases with the removal of an intermediate stop following the triangular inequality in Assumption~\ref{assumption:tria}. Based on this property, we have the subsequent proposition that allows for solving the K-MCR generation problem by sequentially identifying the highest passenger coverage route and generating the mutually disjoint MCRs.
\begin{prop}
Let $r_1,r_2,...,r_K$ be the set of mutually disjoint routes generated as sequentially identified MCR, then $\mathcal{R}_K=\{r_1,r_2,...,r_K\}$ is the K-MCR set that gives the highest passenger coverage. 
\end{prop}
\begin{proof}
The proposition can be proved by induction. First, we verify that the proposition holds for $K=1$. Assuming that $\mathcal{R}_{N-1}$ is the (N-1)-MCR, and $r_N$ is the mutually disjoint MCR identified at step $N$ that shares no overlapping stops with the routes in $\mathcal{R}_{N-1}$. Let us consider that $\mathcal{R}_{N-1}\bigcup \{r_N\}$ is not the N-MCR, then there must be a route $r'\neq r_N$ that yields higher passenger coverage than $r_N$. But this contradicts the fact that $r_N$ is a mutually disjoint MCR. Hence $\mathcal{R}_{N}\bigcup\{r_N\}$ gives the N-MCR.\\
\end{proof}

On the other hand, finding the mutually disjoint route with maximum passenger coverage itself is still nontrivial. In the worst case, one may have to enumerate all possible routes in the network and select the route with the highest passenger coverage. However, the time constraint in Problem~\ref{eq:opt_route} introduces a promising direction for reducing the search space for finding MCR. In particular, starting from a stop $i$, the MCR can be obtained by tracking the potential passenger coverage of the set of stops that are reachable to and from the hub. And at stop $i$, we define \textbf{potential passenger coverage} of a stop $j$ as the largest possible passenger demand that can be covered by extending the route from stop $i$ to stop $j$. Let $\Gamma(i)$ denote the operator for the maximum potential passenger coverage of a route starting from location $i$, we can express the MCR generation as a dynamic programming problem following:
\begin{equation}
    \Gamma(i)=q_i+ \text{max}_{j\in Reach(i)} \Gamma(j)
    \label{eq: dp_equation}
\end{equation}
In the equation, $Reach(i)$ represents the set of stops that are reachable from stop $i$, where $j$ is said to be \textbf{reachable} from $i$ if the travel time of the extended route from $i$ to $j$ will not exceed $\lambda$ times the corresponding shortest travel time as shown in Problem~\ref{eq:opt_route}. Instead of enumerating all routes and select the best one, we can therefore build the solution algorithm for sequentially identifying
mutually disjoint MCR following equation~\ref{eq: dp_equation}. We note that the MCR generation can still be expensive with the worst-case complexity of $\mathcal{O}((e-1)N!)$ as shown in Proposition~\ref{prop:complexity}. However, in practice, the $Reach(i)$ may be significantly smaller in each step due to the travel time constraints and the exact MCR can still be generated efficiently with reasonable values of $\lambda$ (e.g., $\lambda<=1.4$ which represents a 40\% deviation from the fastest alternate mode). We also observe the exact algorithm to generate K-MCR efficiently in our NYC case study. As a consequence, we can tackle the K-MCR problem by utilizing the hub-based structure and the travel time constraints, and the complete algorithm with details for the potential calculation and reachable set identification are summarized in Algorithm~\ref{alg:det}. The algorithm starts with all stops in the study area, finds the MCR by constructing the recursion tree based on the reachable set, and then sets the potential of the stops that are covered in the MCR to 0. This ensures that the routes generated each step are mutually disjoint MCRs. The algorithm terminates when all K routes are generated or can be terminated early if all stops have been visited. And the algorithm can be applied for passengers to and from the hubs separately, with the only difference being that the corresponding travel cost $w_{i,j}$ to be switched to $w_{j,i}$ and that the $t_j^m$ to be set to minimal travel time to and from the hub respectively. 

\begin{prop}
The worst case time complexity for identifying the MCR with $N$ stops is $\mathcal{O}((e-1)N!)$.
\label{prop:complexity}
\end{prop}

\begin{proof}
The complexity can be shown based on equation~\ref{eq: dp_equation}. If we denote the computation requirement for $N$ nodes as $f(N)$, then we have
\begin{equation}
    f(N)=c+(N-1)f(N-1)
\end{equation}
where $c$ is small constant cost at each step, and the equation can be generalized to
\begin{equation}
    f(N)=\sum_{k=0}^{N-2} \frac{(N-1)!}{(N-1-k)!}c
\end{equation}
where 
\begin{equation}
    \lim_{N\rightarrow\infty}\sum_{k=0}^{N-2} \frac{1}{(N-1-k)!}=e-1
\end{equation}
Consequently we have
\begin{equation}
    f(N)_{N\rightarrow\infty}=c(N-1)!(e-1)=\mathcal{O}(N!(e-1))
\end{equation}
\end{proof}

\begin{algorithm}[h!]
  \caption{Hub-based Route Generation Algorithm}
    \begin{algorithmic}[1]
    \Require
    $G=(N,E,Q,W,\lambda)$: network of study area, $K$: number of routes to be generated, $h$: hub location, $\lambda$: deviation threshold
    \Ensure
    $\mathcal{R}_K$ the set of candidate operation routes.
    \For{i=1,2,...,K}
        \State $L_i\leftarrow 0$
        \State $N_h \leftarrow Reach(h,N,Q,L_i,\lambda)$
        \State $P_i,r_i\leftarrow MaxP(h,N_h,L_i,Q,W,\lambda)$
        \State $Q\leftarrow UpdateP(r_i,Q)$
        \State Add $r_i$ to $\mathcal{R}_K$
    \EndFor\\
    \Return{$\mathcal{R}_K$}
    \end{algorithmic}
    \begin{algorithmic}[1]
    \Function{Reach}{$i,V,W,L,\lambda$}
       \State $N_i\leftarrow \emptyset$
       \For{$j\in V$}
        \If{$L+w_{i,j}< \lambda \min t_{j}^m$}
            \State $N_i\leftarrow N_i\bigcup \{j\}$
        \EndIf
        \EndFor\\
       \Return{$N_i$}
    \EndFunction
    \Function{MaxP}{$i,N_i,L_i,Q,W,\lambda$}
    \If{$|N_i|= 1$}
    \Return{$q_j, {j},j\in N_i$}
    \ElsIf{$|N_i|= \emptyset$}
    \Return{$0, \emptyset$}
    \Else
        \State $P\leftarrow 0,r=\emptyset$
           \For{$j\in N_i$}
                \State $N_j\leftarrow$Reach($j,N_i,W,L_i+w_{i,j},\lambda$)
                \State $p,r'\leftarrow$ MaxP($j,N_j,L_i+w_{i,j},Q,W,\lambda$)
                \If{$p+q_i>P$}
                \State $P\leftarrow p+q_i, r\leftarrow r'\bigcup {i} $
                \EndIf
            \EndFor
            \State \Return{$P,r$}
    \EndIf
    \EndFunction
    \Function{UpdateP}{$r,Q$}
    \For{$i\in r$}
    \State $q_i\leftarrow 0$
    \EndFor\\
    \Return{$Q$}
    \EndFunction
    \end{algorithmic}
  \label{alg:det}
\end{algorithm}
\subsection{Heuristic algorithm}
In the case when the exact algorithm may not find the solution in a timely manner, such as when passengers are insensitive to the travel time and a high $\lambda$ value is required for certain real-world applications, we can further develop a heuristic solution approach to identify the near optimal K-MCR based on the following two properties for the route generation problem. 
\begin{prop}
The induced reachable set of a stop is a subset of the reachable set of its predecessor stop. 
\label{prop:reach}
\end{prop}

\begin{prop}
The potential passenger coverage of a stop $i$ is no larger than the sum of the passenger demand of the stops in its reachable set $Reach(i)$.
\label{prop:maxp}
\end{prop}

Proposition~\ref{prop:reach} holds directly from the triangular inequality constraint for travel time, as in Assumption~\ref{assumption:tria}, which states that the set of stops that can not be reached from the predecessor of stop $i$ is also not reachable from stop $i$. Proposition~\ref{prop:maxp} sets the upper bound of the maximum possible potential of a stop since a HHMoD route can not serve more than the number of stops in the current reachable set. Based on these two propositions, we can relax the $\Gamma$ operator in equation~\ref{eq: dp_equation} by replacing it with a heuristic operator $\Gamma^h$ as the sum of the node potential in its reachable set:

\begin{equation}
\Gamma^h(i)=q_i+\Gamma^h(\argmax_{j\in Reach(i)} \sum_{m\in Reach(j)} q_m)   
\end{equation}

This step is equivalent to relaxing the $MaxP$ function in Algorithm~\ref{alg:det} by replacing the recursion with a heuristic $MaxP$ function that calculates the sum of the potential of stops in the reachable set. In this regard, the heuristic approach represents a greedy approach in building the MCR route, and the computation time of the heuristic method is $O(N^2)$. This allows us to find near-optimal K-MCR for large instances in polynomial time. We summarize the details of the heuristic $MaxP$ function in Algorithm~\ref{alg:relax_maxp}.

\begin{algorithm}[htbp!]
  \caption{Heuristic MaxP}
    \begin{algorithmic}[1]
    \Function{Heuristic-MaxP}{$i,N_i,L_i,Q,W,\lambda$}
    \If{$|N_i|= 1$}
    \Return{$q_j, {j},j\in N_i$}
    \ElsIf{$|N_i|= \emptyset$}
    \Return{$0, \emptyset$}
    \Else
        \State $P\leftarrow 0,r=\emptyset$
           \For{$j\in N_i$}
                \State $N_j\leftarrow$Reach($j,N_i,W,L_i+w_{i,j},\lambda$)
                \State $p_j \leftarrow \sum_{k\in N_j} q_k$
            \EndFor
            \State $j'=\argmax_j \,p_j$
            \State $p,r'\leftarrow \text{Heuristic-MaxP}(j',N_{j'},L_i+w_{i,j'},Q,W,\lambda)$
            \State $P\leftarrow p+q_i, r\leftarrow r'\bigcup \{i\} $
    \EndIf
    \State \Return{$P,r$}
    \EndFunction
    \end{algorithmic}
  \label{alg:relax_maxp}
\end{algorithm}

\subsection{Route generation with connections to existing transportation systems}
The effectiveness of the HHMoD can be further improved if the planned service can be aligned with the existing transportation facilities such as the subway and buses. This is especially appealing since existing transit facilities have already established a wide coverage of urban networks. In this regard, the HHMoD may avoid producing overlapping travel segments, and in many cases, the travel time using the subway can be shorter than the ground traffic. And it is likely to realize a higher passenger coverage with the same number of routes. In the following paper, we term \textbf{HHMoD connect} as the HHMoD service whose routes are planned by incorporating the connections to existing transit facilities, and we use \textbf{HHMoD only} to denote the planning of the HHMoD without such connections.    

With the K-MCR generation process described in the previous sections, the framework can be directly extended to embed the connection to transit facilities by modifying the MaxP and Reach functions. And additional inputs are required to describe the costs of such connections. Let $\mathcal{S}$ be the set of stops of the existing transportation facilities.  We denote $d_{i,j}$ as the connection cost between stop $i$ of the HHMoD and the stop $j\in\mathcal{S}$ and $t_{i,j}$ as the known travel time for $i,j\in\mathcal{S}$. We can then  define that a transfer at stop $i$ for the passengers who ride the HHMoD to stop $k$ is feasible if:
\begin{equation}
    L_i+d_{i,j}+t_{j,j'}+d_{j',k}+\sigma \leq \lambda \min_m t_k^m, j,j'\in\mathcal{S}
    \label{eq:transfer}
\end{equation}
Equation~\ref{eq:transfer} states that the summation of travel time for (1) riding HHMoD from the hub to stop $i$ (or from stop $i$ to the hub), (2) walking distance from $i$ to $j$, (3) an inconvenience cost $\sigma$ such as the extra waiting time, (4) the trip time from $j$ to $j'$, and (5) the walking distance from $j'$ to stop $k$ should be no greater than the relaxed shortest travel time $\lambda \min_m t_k^m$. Such a definition is similar to that in equation~\ref{eq:time_constraint}, and subsequently, this constraint can be used to guide the search of the reachable set for the HHMoD connect problem. We note, however, that only one transfer is considered in our problem to avoid unnecessary complications. While considering multiple transfers may further the effectiveness of the route generation process, the improvement is likely to be minor as the addition of extra inconvenience costs will render most of the routes with multiple transfers infeasible. 

Given the transfer costs at each of the HHMoD stops, we can formulate the dynamic programming problem for the HHMoD connect (with $\Gamma^C(\cdot)$ as the max coverage operator) as:
\begin{equation}
    \Gamma^C(i)=q_i+\sum_{k\in CReach(i)} q_k+\max_{j\in HReach(i)} \Gamma^C(j)
    \label{eq:connect}
\end{equation}
Different from equation~\ref{eq: dp_equation}, we now have both $CReach(i)$ and $HReach(i)$ at stop $i$, with $CReach(i)$ being the set of stops that can be reached from stop $i$ through the connections to the transportation systems. As compared to HHMoD only, we have $Reach(i)\subseteq CReach(i)\bigcup HReach(i)$ and $Reach(i)=HReach(i)$ if $CReach(i)=\emptyset$. This ensures that the K-MCR of HHMoD connect will always achieve greater or equal passenger coverage as compared to the K-MCR of HHMoD only with the same K. Nevertheless, we should also note that equation~\ref{eq:connect} does not constitute an exact solution when considering the connections as we greedily add the potential passenger coverage through connections at the first stop when such a connection is feasible. And the exact solution for HHMoD connect should be obtained by constructing the additional recursions on if each connection $(i,k)$ should be established $\forall\,k\in CReach(i)$. This will result in $2^{|CReach(i)|}$ combinations of $max_{j\in HReach(i)} \Gamma^C(j)$ for each step and will quickly render the problem intractable even for small $\lambda$ values. As such, we limit our discussion to the greedy connection scheme as in equation~\ref{eq:connect} and the corresponding K-MCR generation algorithm for HHMoD connect is summarized in Algorithm~\ref{alg:connect}. In addition, the same modification can also be extended for the heuristic algorithms of HHMoD only problem to incorporate the connections to existing transportation systems as following:

\begin{equation}
    \Gamma^{h,C}(i)=q_i+\sum_{k\in CReach(i)} q_k+\Gamma^{h,C}(\argmax_{j\in HReach(i)} \sum_{m\in HReach(j)\bigcup CReach(j)} q_m)    
    \label{eq:heu_connect}
\end{equation}

\begin{algorithm}[h!]
  \caption{Hub-based Route Generation Algorithm with Connections to Transportation Systems}
    \begin{algorithmic}[1]
    \Require
    $G=(N,E,Q,W,\lambda)$: network of study area, $K$: number of routes to be generated, $h$: hub location, 
    $D$: connection costs, $T$: travel time in transportation systems, $\mathcal{S}$: the set of stations of transportation systems, $\sigma$: inconvenience cost for transfer, $\lambda$ deviation threshold
    \Ensure
    $\mathcal{R}_K$ the set of candidate operation routes.
    \For{i=1,2,...,K}
        \State $L_i\leftarrow 0$
        \State $N^H, N^C \leftarrow ConnectReach(h,N,Q,L_i,D,T,\mathcal{S},\sigma, \lambda)$
        \State $P_i,r_i\leftarrow ConnectMaxP(h,N^H,N^C,L_i,Q,W,D,T,\mathcal{S},\sigma,\lambda)$
        \State $Q\leftarrow UpdateP(r_i,Q)$
        \State Add $r_i$ to $\mathcal{R}_K$
    \EndFor\\
    \Return{$\mathcal{R}_K$}
    \end{algorithmic}
    \begin{algorithmic}[1]
    \Function{ConnectReach}{$i,V,W,L,D,T,\mathcal{S},\sigma,\lambda$}
       \State $N_i^H\leftarrow \emptyset$,$N_i^C\leftarrow \emptyset$
       \For{$j\in V$}
        \If{$\exists k,k'\in \mathcal{S}, L+d_{i,k}+t_{k,k'}+d_{k,j}+\sigma<\lambda \min t_{j}^m$}
            \State $N_i^C\leftarrow N_i^C\bigcup \{j\}$
        \ElsIf{$L+w_{i,j}< \lambda \min t_{j}^m$}
            \State $N_i^H \leftarrow N_i^H\bigcup\{j\}$
        \EndIf
        \EndFor\\
       \Return{$N_i^H,N_i^C$}
    \EndFunction
    \Function{ConnectMaxP}{$i,N_i^H,N_i^C,L_i,Q,W,D,T,\mathcal{S},\sigma,\lambda$}
    \If{$|N_i|= 1$}
    \Return{$q_j, {j},j\in N_i$}
    \ElsIf{$|N_i|= \emptyset$}
    \Return{$0, \emptyset$}
    \Else
        \State $P\leftarrow 0,r=\emptyset$
           \For{$j\in N_i$}
                \State $N_j^H,N_j^C\leftarrow$ConnectReach($j,N_i,W,L_i+w_{i,j},D,T,\mathcal{S},\sigma,\lambda$)
                \State $p,r'\leftarrow$ ConnectMaxP($j,N_j^H,N_j^C,L_i+w_{i,j},Q,W,D,T,\mathcal{S},\sigma,\lambda$)
                \If{$p+q_i+\sum_{k\in N_i^C}q_k>P$}
                \State $P\leftarrow p+q_i+\sum_{k\in N_i^C}q_k, r\leftarrow r'\bigcup {i}\bigcup N^C_i $
                \EndIf
            \EndFor
            \State \Return{$P,r$}
    \EndIf
    \EndFunction
    \Function{UpdateP}{$r,Q$}
    \For{$i\in r$}
    \State $q_i\leftarrow 0$
    \EndFor\\
    \Return{$Q$}
    \EndFunction
    \end{algorithmic}
  \label{alg:connect}
\end{algorithm}

\subsection{Shortest-path based route generation} 
In addition to the developed algorithm for MCR generation, we also introduce a shortest-path based heuristic route generation algorithm in~\cite{pinelli2016data} as the benchmark for comparison. The main idea is to first generate a large collection of candidate routes and then prune routes based on user-specified thresholds. We summarize the details of the benchmark heuristic in the Appendix. 

\subsection{Route combination}
By applying the route generation algorithms for both passengers to and from the hub, we obtain $\mathcal{R}_K^d$ as the MCR for serving passengers traveling to the hub and $\mathcal{R}_K^o$ as the MCR for serving passengers traveling from the hub. The next step is to connect $\mathcal{R}_K^d$ and $\mathcal{R}_K^o$ to form round routes and therefore reduce the amount of vacant mileage for the HHMoD fleet. 

The route combination is conducted following two rules. First, the combination of routes to and from hubs with similar demand coverage should be prioritized, which avoids the unbalanced demand issue and makes the best use of the vehicle capacity. Second, the connecting distance between the end of two routes should not exceed the minimum of the trip distance of the two routes. This helps to avoid the combination of two routes of similar demand level but is distant from each other, where directly returning to the hub is the most economical solution. Based on these two criteria, we formulate a bipartite matching problem for the optimal combination of routes in $\mathcal{R}_K^o$ and $\mathcal{R}_K^d$. In particular, we denote $g_{r_1,r_2}$ as the gap between the demand level of route $r_1\in \mathcal{R}_K^o$ and $r_2\in \mathcal{R}_K^d$. If the connecting travel time exceeds the shorter travel time of the two routes, $g_{r_1,r_2}$ is set to an arbitrary large value to represent infeasible combination. Consequently, the optimal route combination can be obtained by solving the following optimization problem:

\begin{equation}
\begin{aligned}
       \min_x & \sum_{r_1\in\mathcal{R}_K^o}\sum_{r_2\in\mathcal{R}_K^d} x_{r_1,r_2}g_{r_1,r_2} \\
       \text{s.t.} \quad & \sum_{r_1} x_{r_1,r_2}=1, \forall r_2\in \mathcal{R}_K^d\\
       &\sum_{r_2} x_{r_1,r_2}=1, \forall r_1 \in \mathcal{R}_K^o\\
       &x_{r_1,r_2}\in \{0,1\}
\end{aligned}
\label{eq:combination}
\end{equation}
The problem is equivalent to the minimum weight perfect bipartite matching problem and can be solved efficiently using the Hungarian algorithm. Readers may refer to~\cite{kuhn1955hungarian} for implementation details. 
\section{Route scheduling under demand uncertainty}
\begin{table}[h!]
	\setstretch{1.2}
	\centering
	\caption{Table of notation}
	\label{tab:notation}
	\begin{tabular}{p{1.5cm}p{11cm}}
		\hline
		Notation & Description \\ \hline
		$ \mathcal{R}$ & The set of routes.\\
		$ \mathcal{N}$ & The set of bus stops.\\
		$X_{s,i}^{F}$ & Served passenger demand rate (minute$^{-1}$) 
		\textbf{from hub} to station $i$ by vehicle route $s$. \\
		$X_{s,i}^{T}$ & Served passenger demand rate (minute$^{-1}$) from station $i$ \textbf{to hub} by vehicle route $s$. \\
		$L_{i}^{F}$ & Unsatisfied passenger demand rate (minute$^{-1}$) \textbf{from hub} to station $i$.\\
		$L_{i}^{T}$ & Unsatisfied passenger demand rate (minute$^{-1}$) from station $i$ \textbf{to hub}.\\
		$\overline{D}_{i}^{F}$ & Expected passenger arrival rate (minute$^{-1}$) \textbf{from hub} to station $i$.\\
		$\overline{D}_{i}^{T}$ & Expected passenger arrival rate (minute$^{-1}$) from station $i$ \textbf{to hub}.\\
		$y_{s}$ & Number of vehicles assigned to route $s$.\\
		$\kappa_{s}$ & Operation indicator for route $s$. 1 for selected route and 0 otherwise. \\		$\kappa_{s,i}$ & Operation indicator for route $s$ that stops at station $i$. 1 for selected route and 0 otherwise. \\
	    $B$ & Total number of available vehicles for a given planning period. \\
		$h_{s}$ & Time headway (minute) for vehicle route $s$.\\
		$h_{s,i}$ & Time headway (minute) for vehicle route $s$ that stops at station $i$.\\
		$h^{min}$ & The minimal time headway (minute) for vehicle route.\\
		$h^{max}$ & The maximal time headway (minute) for vehicle route.\\
		$c_{o}$ & Cost coefficient for operating an additional vehicle  (\$/(minute)). \\
		$c_{w}$ & Cost coefficient for waiting time of served passengers (\$/(minute)).\\
		$c_{l}$ & Cost coefficient for penalizing an unsatisfied passenger trip (\$).\\ 
		$T_{s}$ & Total round trip travel time (minute) for route $s$. \\
		$\delta_{s,i}^{F}$ & Indicator variable. $\delta_{s,i}^{F}=1$ if route $s$ serves the trip \textbf{from hub} to station $i$, and 0 otherwise.\\
		$\delta_{s,i}^{T}$ & Indicator variable. $\delta_{s,i}^{T}=1$ if route $s$ serves the trip from station $i$ \textbf{to hub}, and 0 otherwise.\\
		$C$ & Capacity of each vehicle.\\
        \hline
	\end{tabular}
\end{table}

Given the set of combined candidate routes $\mathcal{R}$, we next introduce a different set of notations as summarized in Table~\ref{tab:notation} to formulate the route scheduling problem. This is different from the planning of conventional public transportation services.  The users of the HHMoD service are usually more time-sensitive and pay a premium for better mobility that offers shorter travel time and better riding experience, and more importantly, more reliable services. The HHMoD service is vulnerable to the loss of its market if it fails to satisfy the revealed demand that is different from the daily expected levels. In this regard, it is crucial to plan the service in the face of demand uncertainty and ensure the reliability of the scheduled service even under the worst realization of passenger arrivals at the stops. This motivates us to investigate the robust optimization for the route scheduling and identify the optimal strategy that is resilient to worst-case demand realizations. Here, for a given time period (e.g., 30 minutes or an hour), the route scheduling problem is to determine the allocation of a fleet of $B$ vehicles to the set of candidate routes $\mathcal{R}$ and the corresponding frequency of each route, termed by the headway $h_s$, with the objective to minimize the weighted combination of the total operation cost, the waiting time cost for served passengers and the penalty for unsatisfied passenger demand. We assume the Poisson arrival of passengers at each stop and that the fleet is composed of homogeneous vehicles of the same capacity $C$. We consider the headway of each route is scheduled at the per-minute level, and therefore, the headway setting is modeled as an integer variable. We follow the framework in~\cite{furth1981setting,gkiotsalitis2017exact} to develop the mathematical formulation for the HHMoD frequency problem. Before accounting for the demand uncertainty, we first summarize the nominal model (deterministic demand) for the optimal route scheduling problem with individual components in the objective function can be expressed as
\begin{subequations}
\begin{align}
F_{operation}&=c_o \sum_{s\in\mathcal{R}} y_s\\ 
F_{waiting}&=c_w\sum_{s\in\mathcal{R}} 
(\sum_{i\in\mathcal{N}}X_{s,i}^{F}\delta_{s,i}^F\frac{h_{s}}{2} + \sum_{i\in\mathcal{N}}X_{s,i}^{T}\delta_{s,i}^T\frac{h_{s}}{2})\\ 
F_{loss}&=c_l\sum_{i\in\mathcal{N}}(d_i^FL_i^F+d_i^TL_i^T)\label{eq:loss} 
\end{align}
\end{subequations}
Note that for a passenger waiting at the stop, the expected waiting time is equal to half of the vehicle route headway $h_{s}$ as we assume the random arrival of passengers~\cite{gkiotsalitis2017exact}. In equation~\ref{eq:loss}, the travel distance $d_i$ by alternate mode is also included to penalize the loss of services. For instance, we may penalize longer trips more than shorter trips if the HHMoD may fail to serve these trips, which are instead served by taxis and FHVs. With the objective functions, the nominal HHMoD scheduling problem can be formally written as
\begin{equation}
  \tag{Nominal HHMoD Scheduling}
  \min_{y,h,\kappa,X,L}\quad F_{operation}+F_{waiting}+ F_{loss}\label{eq:obj}
\end{equation}
subject to
\begin{subequations}
\begin{align}[left ={\mathcal{H}=  \empheqlbrace}]
    &y_sh_s\geq T_s &&\forall s\in\mathcal{R} \label{eq:service_frequency_1}\\
    &y_s\leq B\kappa_s, \kappa_s\leq y_s &&\forall s\in \mathcal{R} \label{eq:service_frequency_2}\\
    &h^{min}\leq h_s\leq h^{max} &&\forall s\in\mathcal{R} \label{eq:min_freq_1}\\
    &\sum_{s\in\mathcal{R}} y_s\leq B \label{eq:fleet_size_1}\\
    &y_s\in \mathbbm{Z}, h_s\in\mathbbm{Z},\kappa_s\in\{0,1\} &&\forall s\in \mathcal{R} 
\end{align}
\end{subequations}
\begin{subequations}
\begin{align}[left ={\mathcal{F}=  \empheqlbrace}]
    &\sum_{i\in\mathcal{N}} h_sX_{s,i}^{F}\delta_{s,i}^{F}\leq C &&\forall s\in\mathcal{R}\label{eq:to_hub_capacity_1}\\
    &\sum_{i\in\mathcal{N}} h_sX_{s,i}^{T}\delta_{s,i}^{T}\leq C &&\forall s\in\mathcal{R} \label{eq:from_hub_capacity_1}\\
    &\sum_{s\in\mathcal{R}} X_{s,i}^{F}+L_{i}^{F}= \bar{D}_i^{F} && \forall i\in\mathcal{N} \label{total capacity 1}\\
    &\sum_{s\in\mathcal{R}} X_{s,i}^{T}+L_{i}^{T}= \bar{D}_i^{T} && \forall i\in\mathcal{N} \label{total capacity 2}\\
    &X_{s,i}^{F}\leq \kappa_s \bar{D}_i^F, &&\forall s\in\mathcal{R},i\in\mathcal{N} \label{upper1}\\
    &X_{s,i}^{T}\leq \kappa_s \bar{D}_i^T, &&\forall s\in\mathcal{R},i\in\mathcal{N} \label{upper2}\\
    & X_{s,i}^{F}\geq 0, X_{s,i}^{T}\geq 0, L_i^{F}\geq 0, L_i^{T}\geq 0 &&\forall s\in\mathcal{R},i\in\mathcal{N} \label{eq:integer_cons_1}
    \end{align}
\end{subequations}

where $\mathcal{H}$ describes the set of constraints that regulate the number of vehicles $y_{s}$ and headway setting $h_{s}$ for the planned HHMoD service and $\mathcal{F}$ characterizes the set of passengers that can be served by the scheduled services. 
Equation~(\ref{eq:service_frequency_1}) establishes the relationship between assigned vehicle number $y_{s}$, vehicle headway $h_{s}$ and round trip travel time $T_{s}$. For instance, if the the round trip travel time is $T_{s} = 36$ minutes and the required time headway is $h_{s} = 12$ minutes, then we need at least $y_{s} = 3$ buses. Equation~(\ref{eq:service_frequency_2}) implies that $y_{s}$ is a positive integer upper-bounded by $B$ when $\kappa_{s}=1$. If $\kappa_{s}=0$, then $y_{s}=0$. Equation~(\ref{eq:min_freq_1}) shows the upper and lower bound of the time headway $h_{s}$. Equation~(\ref{eq:fleet_size_1}) represents that the total number of assigned vehicle do not exceed to vehicle number constraint $B$. 
Besides, equations (\ref{eq:to_hub_capacity_1}) and~(\ref{eq:from_hub_capacity_1}) imply that the total boarding passengers for routes from and to activity hubs should not exceed the vehicle capacity $C$.
Equations~(\ref{total capacity 1}) and (\ref{total capacity 2}) describe that the sum of satisfied and unsatisfied demand is equal to the total demand $\bar{D}_{i}^F$ and $\bar{D}_{i}^T$ in the city. Finally, equations (\ref{upper1}) and (\ref{upper2}) represent the served demand at stop $i$ does not exceed the total demand of stop $i$. The nominal problem aims at identifying the optimal service headway $h_s$ and assigned vehicles $y_s$ for each route $s\in\mathcal{R}$ with the assignment of served and unsatisfied passengers. \ref{eq:obj} is a nonlinear and nonconvex mixed-integer programming problem, which is difficult to solve directly due to the existence of bilinear constraints $y_{s}h_{s}$ and bilinear terms $h_{s}X_{s,i}^{F}\delta_{s,i},h_{s}X_{s,i}^{T}\delta_{s,i}$ in the objective function.

We note that the \ref{eq:obj} considers known passenger demand $\bar{D}_i^{T}$ and $\bar{D}_i^{F}$ as a point estimation (such as expected value). Nevertheless, for the proactive planning of HHMoD service in real-time, the passenger demand may vary significantly, and the planned service based on expected passenger demand $\bar{D}_i^{T}$ and $\bar{D}_i^{F}$ is unlikely to align with the actual realization of demand profile. This may lead to a situation where there is overly supply on certain routes, with the revealed demand being significantly lower than the scheduled capacity. Meanwhile, the oversupply in certain routes may result in the shortage of supply to serve the actual demand on other routes. In this regard, if the uncertainty set of passenger demand is known, the nominal scheduling problem can be converted into a two-stage robust route scheduling problem to negate the impacts from the mismatch between scheduled supply and realized demand. Specifically, in the first stage, the master problem represents the problem of scheduling vehicle allocation and route headway with known number of passengers. As for the second stage, we have the recourse problem where the worst-case-oriented passenger demand distribution is realized in response to the master problem's route configuration. We consider the box constraint to capture the demand uncertainty as
\begin{equation}
    \sum_{s\in\mathcal{R}} X_{s,i}^{F}+L_i^F=\bar{D}_i^F+p_i^F Z_i^F, \forall\,i\in\mathcal{N}
    \label{ro_con1}
\end{equation}
\begin{equation}
    \sum_{s\in\mathcal{R}} X_{s,i}^{T}+L_i^T=\bar{D}_i^T+p_i^T Z_i^T, \forall\,i\in\mathcal{N} 
    \label{ro_con2}
\end{equation}
where $\bar{D}_i$ is the expected demand level and $Z_i$ is the maximum demand deviation at stop $i$. And we also have:
\begin{equation}
    X_{s,i}^{F}\leq \kappa_s(\bar{D}_i^F+p_i^F Z_i^F), \forall s\in\mathcal{R},i\in\mathcal{N}
    \label{ro_con3}
\end{equation}
\begin{equation}
    X_{s,i}^{T}\leq \kappa_s(\bar{D}_i^T+p_i^T Z_i^T), \forall s\in\mathcal{R},i\in\mathcal{N}
    \label{ro_con4}
\end{equation}
To control the conservatism of the robust optimization outcomes, the budget for demand uncertainty is introduced as in~\cite{bertsimas2003robust} where the decision-maker may decide to account for the maximum deviation of passenger demand for up to $\Gamma$ locations with
\begin{equation}
    \mathcal{P}=\{p|p_i^F, p_i^T\in \{0,1\}, \sum_{i \in V} (p_i^F + p_i^T)=\Gamma\}
\end{equation}
With the above defined demand uncertainty set, the two-stage robust optimization problem (RO) can be formulated as:
\begin{equation}
  \tag{RO-HHMoD Scheduling}
  \begin{aligned}
\min_{y,h,\kappa\in\mathcal{H}}\; & F_{operation}+\max_{p\in\mathcal{P}}\min_{X,L\in \mathcal{F}(p)}(F_{waiting}+F_{loss})\\ 
\end{aligned}  
\label{eq:robust}
\end{equation}
where $\mathcal{F}(p)$ follows from replacing the constraints~\ref{total capacity 1}-~\ref{upper2} by equations~\ref{ro_con1}-\ref{ro_con4}.
\subsection{Recourse problem reformulation}
As the relatively complete recourse property always holds for our recourse problem, and that the inner minimization problem is a linear programming problem if $p$ is supplied, we can convert the max-min recourse problem~\ref{eq:robust} into a single maximization problem by utilizing the strong duality condition with known $h^*$ and $y^*$ from the master problem: 

\begin{eqnarray}
\begin{aligned}
    \max_{p,u^F,u^T,v^F,v^T,\lambda^F,\lambda^T} &
    \sum_{i=1}^{|\mathcal{R}|}C(u_i^F +
    u_i^T) + 
    \sum_{i=1}^{|\mathcal{N}|} \kappa_{s,i}((\bar{D}_i^F+p_i^F Z_i^F)v_{i}^{F} +
    (\bar{D}_i^T+p_i^T Z_i^T)v_{i}^{T})\\
    &+
    \sum_{i=1}^{|\mathcal{N}|}((\bar{D}_i^F+p_i^F Z_{i}^{F})\lambda_i^F + (\bar{D}_i^T+p_i^T Z_{i}^{T})\lambda_i^T)\\
    \text{s.t.}\quad &      \sum_{s \in \mathcal{R}}h_{s}\delta_{s,i}^{F}u_{i}^{F}+v_{i}^{F}+\lambda_{i}^F\leq c_{w}\sum_{s\in \mathcal{R}}\frac{h_{s}}{2}\delta_{s,i}^{F}, i=1,2,...,|\mathcal{N}| \\
       & \sum_{s \in \mathcal{R}}h_{s}\delta_{s,i}^{T}u_{i}^{T}+v_{i}^{T}+\lambda_{i}^T\leq c_{w}\sum_{s\in \mathcal{R}}\frac{h_{s}}{2}\delta_{s,i}^{T}, i=1,2,...,|\mathcal{N}| \\
    &v_{i}^{F}\leq c_{l}d_i^F, i=1,2,...,|\mathcal{N}| \\
    &v_{i}^{T}\leq c_{l}d_i^T, i=1,2,...,|\mathcal{N}| \\
    &\sum_{i=1}^N (p_i^F+p_i^T)=\Gamma\\
    &u_i^F, u_i^T\leq 0, i =1,2,...,|\mathcal{R}|\\
    &v_i^F, v_i^T\leq 0, i =1,2,...,|\mathcal{N}|\\
    &\lambda_i^F,\lambda_i^T \;free, i =1,2,...,|\mathcal{N}|
    \label{dual_cons}
    \end{aligned}
\end{eqnarray}


With $u,v,\lambda$ and $p$ being the decision variables, the converted recourse problem now becomes a bilinear integer programming problem with the bilinear terms $\lambda_ip_i$ and $v_ip_i$. As $p_i^F$ and $p_{i}^T$ are binary variables, instead of solving a nonconvex problem, we can further convert this problem into an equivalent mixed integer linear programming problem (MILP) by making use of the big-M method to replacing each $p_i^F\lambda_i^F$ with $\gamma_i^F$, $p_{i}^T\lambda_i^T$ with $\gamma_i^T$, $p_{i}^{F}v_{i}^{F}$ with $\psi_{i}^{F}$, $p_{i}^{T}v_{i}^{T}$ with $\psi_{i}^{T}$ as:
\begin{eqnarray}
\begin{aligned}
    \max_{p,u^F,u^T,v^F,v^T,\lambda^F,\lambda^T} &
    \sum_{i=1}^{|\mathcal{R}|}C(u_i^F +
    u_i^T) + 
    \sum_{i=1}^{|\mathcal{N}|} \kappa_{s,i}(\bar{D}_i^F v_{i}^{F}
    + Z_{i}^{F} \psi_{i}^{F} + \bar{D}_i^T v_{i}^{T} + Z_{i}^{T} \psi_{i}^{T})\\
    &+\sum_{i=1}^{|\mathcal{N}|}(\bar{D}_i^F\lambda_i^F+\gamma_{i}^{F}Z_{i}^{F}
    +\bar{D}_i^T\lambda_i^T+\gamma_{i}^{T}Z_{i}^{T})\\
    \text{s.t.}\quad & \sum_{s \in \mathcal{R}}h_{s}\delta_{s,i}^{F}u_{i}^{F}+v_{i}^{F}+\lambda_{i}^F\leq c_{w}\sum_{s\in \mathcal{R}}\frac{h_{s}}{2}\delta_{s,i}^{F}, i=1,2,...,|\mathcal{N}| \\
       & \sum_{s \in \mathcal{R}}h_{s}\delta_{s,i}^{T}u_{i}^{T}+v_{i}^{T}+\lambda_{i}^T\leq c_{w}\sum_{s\in \mathcal{R}}\frac{h_{s}}{2}\delta_{s,i}^{T}, i=1,2,...,|\mathcal{N}| \\
       &v_{i}^{F}\leq c_{l}, i=1,2,...,|\mathcal{N}| \\
    &v_{i}^{T}\leq c_{l}, i=1,2,...,|\mathcal{N}| \\
    &\sum_{i=1}^N (p_i^F+p_i^T)=\Gamma\\
    &\gamma_i^F \leq p_i^F M,\;\gamma_i^T \leq p_i^T M,\;i=1,2,...,|\mathcal{N}|\\
    &\gamma_i^F \leq \lambda_i^F,\;\gamma_i^T \leq \lambda_i^T,\;i=1,2,...,|\mathcal{N}|\\
    &\psi_i^F \leq p_i^F M,\;\psi_i^T \leq p_i^T M,\;i=1,2,...,|\mathcal{N}|\\
    &\psi_i^F \leq v_i^F,\;\psi_i^T \leq v_i^T,\;,i=1,2,...,|\mathcal{N}|\\
    &u_i^{F},u_{i}^{T}\leq 0,\;v_i^{F},v_{i}^{T}\leq 0,\;\gamma_i^{F},\gamma_{i}^{T} \geq 0,\psi_i^{F},\psi_{i}^{T} \geq 0\;,i=1,2,...,|\mathcal{N}|
    \end{aligned}
    \label{eq:MILP_recourse}
\end{eqnarray}
and Problem~\ref{eq:MILP_recourse} can therefore be solved to global optimum using off-the-shelf commercial solvers. 

\subsection{Master problem reformulation}
We can obtain the optimal value of $p^*$ by solving the reformulated recourse problem, which allows us to rewrite an equivalent master problem by introducing the auxiliary variable $\eta$ as:
\begin{equation}
  \begin{aligned}
\min_{y,h,\kappa, X,L}\; & F_{operation}+\eta\\
s.t. \quad \quad & \eta\geq F_{waiting}+F_{loss}\\ 
&(y,h,\kappa)\in\mathcal{H}\\
&(X,L)\in\mathcal{F}(p^*)
\end{aligned}  
\end{equation}
To solve the master problem, the major challenge arises from the bilinear term $h_{s}X_{s,i}$ in the objective function and the constraints, and the bilinear term $y_sh_s$ in the constraint, where both $y_s$ and $h_s$ are integer variables. As both bilinear terms share the common elements $h_s$ as an integer variable, we introduce the unary expansion to express $h_s$ as the summation of a series of binary variables $g_s^k$ as:
\begin{equation}
    h_s=\sum_{k=0}^H 2^k g_s^k, g_s^k \in \{0,1\}, \forall s
\end{equation}
with $H=\lc log_2(h^{max}+1) \rc -1$. We therefore have 

\begin{equation}
y_sh_s=\sum_{k=0}^H 2^kg_s^ky_s,\forall s
\label{eq:unary}
\end{equation}
and we can further replace $g_s^ky_s$ with $z_s^k$ and derive the boundary for $z_s^k$ based on the McCormick's envelope:
\begin{eqnarray}
\begin{aligned}
    y_sh_s&=\sum_{k=0}^H 2^kz_s^k \\
    g_s^ky_s^{min}&\leq z_s^k\leq g_s^ky_s^{max}\\
    y_s-y_s^{max}(1-g_s^k)&\leq z_s^k \leq y_s-y_s^{min}(1-g_s^k)
    \label{eq:mccormick_1}
\end{aligned}
\end{eqnarray}
We note that the McCormick envelope constructed for $z_s^k$ is an exact relaxation due to the binary nature of $g_s^k$, and hence equations~equations~(\ref{eq:unary}) and~equations~(\ref{eq:mccormick_1}) result in an exact relaxation of the binary constraint $y_sh_s$. Similarly, we can convert the bilinear terms $h_sX_i$ with $g_s^kX_i$ and replace each $g_s^kX_i$ with $q_{s,i}^{k}$ as:
\begin{eqnarray}
\begin{aligned}
    &h_sX_{s,i}^{F}=\sum_{k=0}^H 2^kq_{s,i}^{k,F} \\
    &h_sX_{s,i}^{T}=\sum_{k=0}^H 2^kq_{s,i}^{k,T}\\
    &0 \leq q_{s,i}^{k,F}\leq g_s^{k}(\bar{D}_i^F+p_i^{F}Z_i^{F})\\
    &0 \leq q_{s,i}^{k,T}\leq g_s^{k}(\bar{D}_i^T+p_i^{T}Z_i^{T})\\
    &X_{s,i}^{F}-(\bar{D}_i^{F}+p_i^{F,l}Z_i^F)(1-g_s^k)\leq q_{s,i}^{k,F} \leq X_{s,i}^{F}\\
    &X_{s,i}^{T}-(\bar{D}_i^{T}+p_i^{T,l}Z_i^T)(1-g_s^k)\leq q_{s,i}^{k,T} \leq X_{s,i}^{T}\\
    \label{eq:mccormick}
\end{aligned}
\end{eqnarray}
The exact relaxations in equations~\ref{eq:mccormick_1} and~\ref{eq:mccormick} convert the bilinear master problem into an equivalent MILP problem so that the master problem can also be solved to global optimality using the off-the-shelf MILP solvers. 
\subsection{Solving the two-stage robust route scheduling problem}
With the reformulation of both the master and the recourse problems, we are now ready to tackle the two-stage robust route scheduling problem. It is in general difficult to solve the two-stage RO problems due to its multilevel optimization structure, and there are in general two widely-adopted solution approaches in the literature: the extension of Bender's decomposition method~\cite{bertsimas2012adaptive} and the Column and Constraint Generation (C\&CG) method~\cite{zeng2013solving}. In this study, we adopt the C\&CG method as it has been shown to achieve superior convergence performances than the Bender's decomposition method for practical sized problems (e.g., hundreds of times faster in~\cite{an2014reliable} with guaranteed algorithm convergence as both our reformulated master and recourse problems can be solved to global optimum. The C\&CG algorithm solves the two-stage RO by iterating over the master and recourse problems, where optimality cuts associated with primal variables are added to the master problem until convergence. 

Let $\Tau=\{p^1,p^2,...,p^L\}$ be the set of realized uncertainty coefficients up to iteration $L$, the corresponding master problems follows:
\begin{subequations}
    \begin{align}
    \min_{y,z,g,q,\kappa,X,L} \quad & c_o\sum_{s\in\mathcal{R}}y_s+\eta\\
    &\eta \geq \frac{1}{2}c_w \sum_{s\in\mathcal{R}}
    \sum_{k=0}^H\sum_{i \in \mathcal{N}} 2^k(\delta_{s,i}^{F}q_{s,i}^{k,F,l}+\delta_{s,i}^{T}q_{s,i}^{k,T,l})+ c_l\sum_{i\in V}(d_i^FL_i^{F,l}+d_i^TL_i^{T,l}), l=1,2,3,...,L \label{con_first}\\
    &y_s\leq B\kappa_s, \kappa_s\leq y_s,\forall s\in\mathcal{R}\\
    &\sum_{s\in\mathcal{R}} y_s\leq B \label{eq:fleet_size}\\
    &\sum_{k=0}^H 2^{k} z_s^k\geq T_s, \forall s\in\mathcal{R} \label{eq:service_frequency}\\
    &\sum_{i\in \mathcal{N}} \delta_{s,i}^{F}\sum_{k=0}^H 2^k q_{s,i}^{k,F,l}\leq C, \, l=1,2,3,...,L, \forall s\in\mathcal{R}\\
    &\sum_{i\in \mathcal{N}} \delta_{s,i}^{T}\sum_{k=0}^H 2^k q_{s,i}^{k,T,l}\leq C, \, l=1,2,3,...,L, \forall s\in\mathcal{R}\\
    &\sum_{s\in\mathcal{R}} \delta_{s,i}^{F} X_{s,i}^{F,l} + L_i^{F,l}=\bar{D}_i^{F}+p_i^{F,l}Z_i^{F},\,\forall i\in\mathcal{N},l=1,2,3,...,L\\
    &\sum_{s\in\mathcal{R}}  \delta_{s,i}^{T} X_{s,i}^{T,l} + L_i^{T,l}=\bar{D}_i^{T}+p_i^{T,l}Z_i^{T},\,\forall i\in\mathcal{N},l=1,2,3,...,L\\
    &X_{s,i}^{F,l} \leq \kappa_{s}(\bar{D}_i^{F}+p_i^{F,l}Z_i^{F}) ,\forall i \in \mathcal{N},\forall s\in\mathcal{R},l=1,2,3,...,L\\
    &X_{s,i}^{T,l} \leq \kappa_{s}(\bar{D}_i^{T}+p_i^{T,l}Z_i^{T}) ,\forall i \in \mathcal{N},\forall s\in\mathcal{R},l=1,2,3,...,L\\
    &0\leq z_s^k\leq g_s^kB\\
    &y_s-B(1-g_s^k)\leq z_s^k \leq y_s\\
    &0 \leq q_{s,i}^{k,F,l}\leq g_s^{k}(\bar{D}_i^F+p_i^{F,l}Z_i^{F}), l=1,2,3,...,L\\
    &0 \leq q_{s,i}^{k,T,l}\leq g_s^{k}(\bar{D}_i^T+p_i^{T,l}Z_i^{T}), l=1,2,3,...,L\\
    &X_{s,i}^{F,l}-(\bar{D}_i^{F}+p_i^{F,l}Z_i^F)(1-g_s^k)\leq q_{s,i}^{k,F,l} \leq X_{s,i}^{F,l},
    l=1,2,3,...,L\\
    &X_{s,i}^{T,l}-(\bar{D}_i^{T}+p_i^{T,l}Z_i^T)(1-g_s^k)\leq q_{s,i}^{k,T,l} \leq X_{s,i}^{T,l},
    l=1,2,3,...,L \label{con_last}\\
    &y_s\in\mathbbm{Z}, z_{s}^{k}\in\mathbbm{Z},g_s^k\in\{0,1\},\kappa_s \in\{0,1\} \label{int_variable}\\ 
    &q_s^{k,F,l},q_s^{k,T,l}\geq 0, X_{s,i}^{F,l},X_{s,i}^{T,l}\geq 0,
    L_{i}^{F}, L_{i}^{T} \geq 0 \label{con_variable} 
    \end{align}
    \label{master_p}
\end{subequations}
With the master problem above, we summarize the major steps for implementing the C\&CG algorithm to solve the two-stage robust route scheduling problem as following:

\begin{steps}
    \item Initialization: Set $LB=-\inf$, $UB=\inf$. Set $N=0$, set $\Tau=\emptyset$.
    \item Find an initially feasible $p^0\in\mathcal{P}$, add $p^0$ to $\Tau$. Set $N=1$.
    \item Solve master problem~(\ref{master_p}) to global optimum and obtain $y^{*},h^{*},\kappa^{*},\eta^{*}$. Update $LB=c_o\sum_{s\in\mathcal{R}}y^*_s+\eta^*$.
    \item With $y^{*},h^{*},\kappa^{*}$, solve the recourse problem (\ref{eq:MILP_recourse}) and obtain $p^{N}$. Update $UB=c_o\sum_{s\in\mathcal{R}}y^*_s+ F_{waiting}(X^*,L^*) + F_{loss}(X^{*},L^{*})$. 
    \item If $UB-LB\leq \epsilon$, terminate and return $y^*$ and $h^*$ as the optimal solution. Otherwise update $N=N+1$, add $p^{N}$ to $\Tau$ and go to Step 3.  
\end{steps}
As for Step 2, an initially feasible $p^0$ can be easily selected by randomly select $\Gamma$ stops and set the corresponding $p_i$ to 1. In this study, the Gurobi 9.0 optimizer is used to solve the MILP for both master and recourse problems. We note that, with the unary expansion and the relaxation of the bilinear terms, the C\&CG algorithm will add $2|\mathcal{N}|(H+1)$ continuous variables and $2|\mathcal{R}|+8|\mathcal{N}|+1$ constraints into the reformulated master problem. While this will quickly enlarge the size of the master problems, the additional iterations do not increase the number of integer and binary variables and fast algorithm convergence is still observed through our numerical experiments where the algorithm will usually terminate within a few iterations. 

%
\section{Results}
\subsection{Experiments setting}
 We choose NYC as the study area and demonstrate the effectiveness of the proposed HHMoD framework by developing the HHMoD service at the JFK airport, which is associated with the origins and destinations of over 3.5\% of total daily FHV and taxi demand in 2018. This section presents three sets of experiments: (1) the computing performances of the route generation and robust planning models, (2) the effectiveness of the route generation approaches, and (3) the resulting costs of the fleet arrangements through robust optimization. We prepare four demand scenarios during weekdays for both the \textbf{HHMoD only} services and \textbf{HHMoD connect} services: AM peak scenario from 7:00-9:00, off-peak scenario from 13:00-15:00, PM peak scenario from 17:00-19:00, and night time scenario from 21:00-23:00. 
 
To prepare model inputs and quantify the modeling parameters, we collect information from publicly available datasets, including road information, geographical subdivisions of the city, public transit information, and demand of taxi and for-hire vehicles (FHV). The network of our study area is built from the taxi zone shapefile of NYC, which has 263 taxi zones across five major boroughs. State Island is not included due to a low demand level (less than 0.5\% of total demand to hubs such as LGA, JFK, and Penn Station). Due to the lack of pick up and drop off locations, we assume the centroid of each zone being the stops for the HHMoD service and we use the combination of FHV and taxi demand to and from the JFK airport in each zone as the demand for each stop based on the 2018 trip data~\cite{taxidata2018}. A more realistic scenario can be prepared by conducting spatial clustering to determine the number of stops, and the associated demand level when the pickup and drop off locations are available. We only select the top 100 highest demand stops to and from the hubs, which account for over 85\% of the total demand during each time intervals. This selection results in the K-MCR problem with 100 stops each way and the robust optimization problems with the budget for demand uncertainty for up to 200 stops. We calibrate the average demand for each stop as well as the maximum deviation at each stop using the 2018 NYCTLC taxi and FHV trip data~\cite{taxidata2018}. The shortest travel time between stops is calibrated by the average FHV and taxi trip time between the two areas. For HHMoD connect, we consider the planned HHMoD being connected to the NYC subway system of 36 lines and 247 stations. The walking distance is measured between each HHMoD stop and subway station and then converted to the corresponding walking time. The trip time between subway stations is obtained from the NYC subway timetable, and an inconvenience cost of 500 seconds is added to all trips that require a transfer between the HHMoD and the subway system in light of the additional waiting and delay. Finally, we summarize the modeling parameters that are required for the route generation and fleet scheduling problems in Table~\ref{tab:params}. 

\begin{table}[h!]
\caption{Parameter setting}
\centering
\begin{tabular}{ll}
\toprule
Name                   & Value                           \\
\hline
$c_o$                   & 50 (\$)                         \\
$c_w$                   & 0.5 (\$/minute)                 \\
$c_l$                   & 5 (\$/mile)                    \\
$\lambda$ & 1.3 (default for route scheduling unless otherwise specified) \\
$\sigma$ & 500 (default for route scheduling unless otherwise specified)\\
C & 20\\
B & 200 \\
$h^{min}$ & 3 (minute)\\
$h^{max}$ & 30 (minute)\\
\bottomrule
\end{tabular}
\label{tab:params}
\end{table}

\subsection{Route generation}
We first demonstrate the computational performances for the route generation, and the results are summarized in Figure~\ref{fig:cpu_routes}, where we vary the time threshold value $\lambda$ and compare how the computational time scales with increasing time threshold. An inconvenience cost of 500 seconds is added to the HHMoD connect approach. And the results reported here are for the computational time required to cover all passenger demand in the study area. For both HHMoD connect and HHMoD only cases, the heuristic methods can generate the K-MCR in only a few seconds. Whereas in the exact method, the computational time for route generation scales exponentially with increasing time threshold since a higher threshold will result in more number of accessible nodes in the reachable set at each stop. As for HHMoD connect, we observe that it is much more efficient than the HHMoD only scenario, though the computational time still scales exponentially. The reason is that, by connecting to the subway system, certain candidate stops in the reachable set can be directly included in the final solution as long as they can be reached through a transfer to the existing transit system. This helps to reduce the number of nodes to be further branched on, thus saving computational time. However, the extent of the reduction that can be achieved may vary depending on the time threshold of $\lambda$, the inconvenience cost for transfers, and also the layout of the existing transit systems. In Figure~\ref{fig:cpu_routes_incov}, we also visualize how the computational time may change with inconvenience cost ranging from 100 seconds to 1200 seconds while holding the $\lambda=1.3$. We observe that the increase in computation time for HHMoD connect has a superlinear relationship concerning the increase in inconvenience cost until it reaches the value of around 900 seconds. Beyond this point, the transfer to the subway system is no longer viable, and the computational time is equivalent to that of the HHMoD only case. On the other hand, we report that it will take more than 30 minutes for the shortest path based heuristic to generate the set of candidate paths, and the generated path is unable to fulfill all the passenger demand in the study area (as can be seen in Figure~\ref{fig:coverage_routes_compare}). In this regard, the shortest-path-based heuristic may not apply to the route generation problem in our case. Instead, the developed exact approaches can generate high-quality solutions efficiently with reasonable $\lambda$ values. In the case where a high $\lambda$ is required, one can use the developed heuristic approaches to generate the K-MCR greedily in real-time. 

\begin{figure}[h!]
    \centering
    \subfloat[AM peak]{\includegraphics[width=0.33\linewidth]{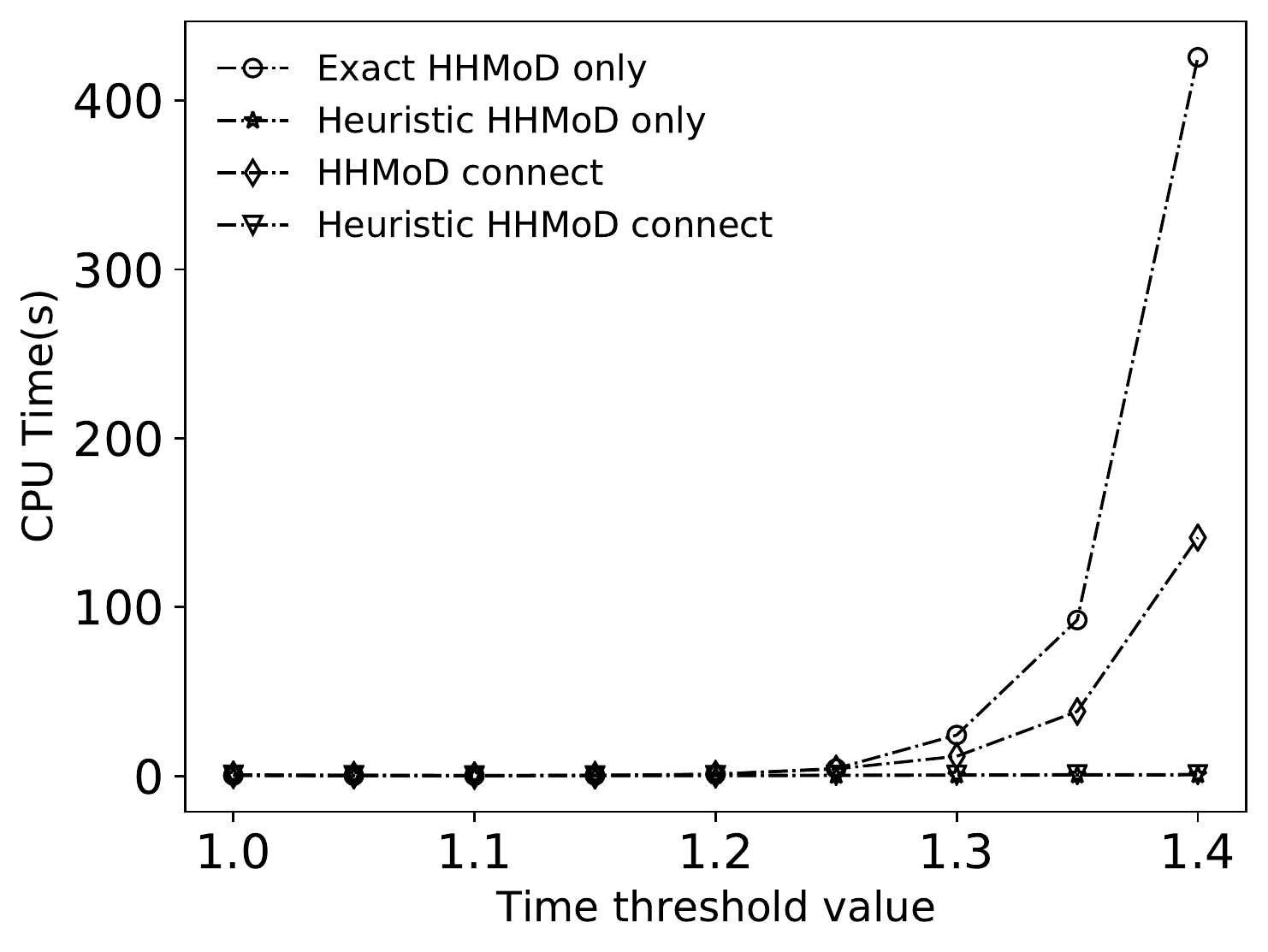}}  
    \subfloat[PM peak]{\includegraphics[width=0.33\linewidth]{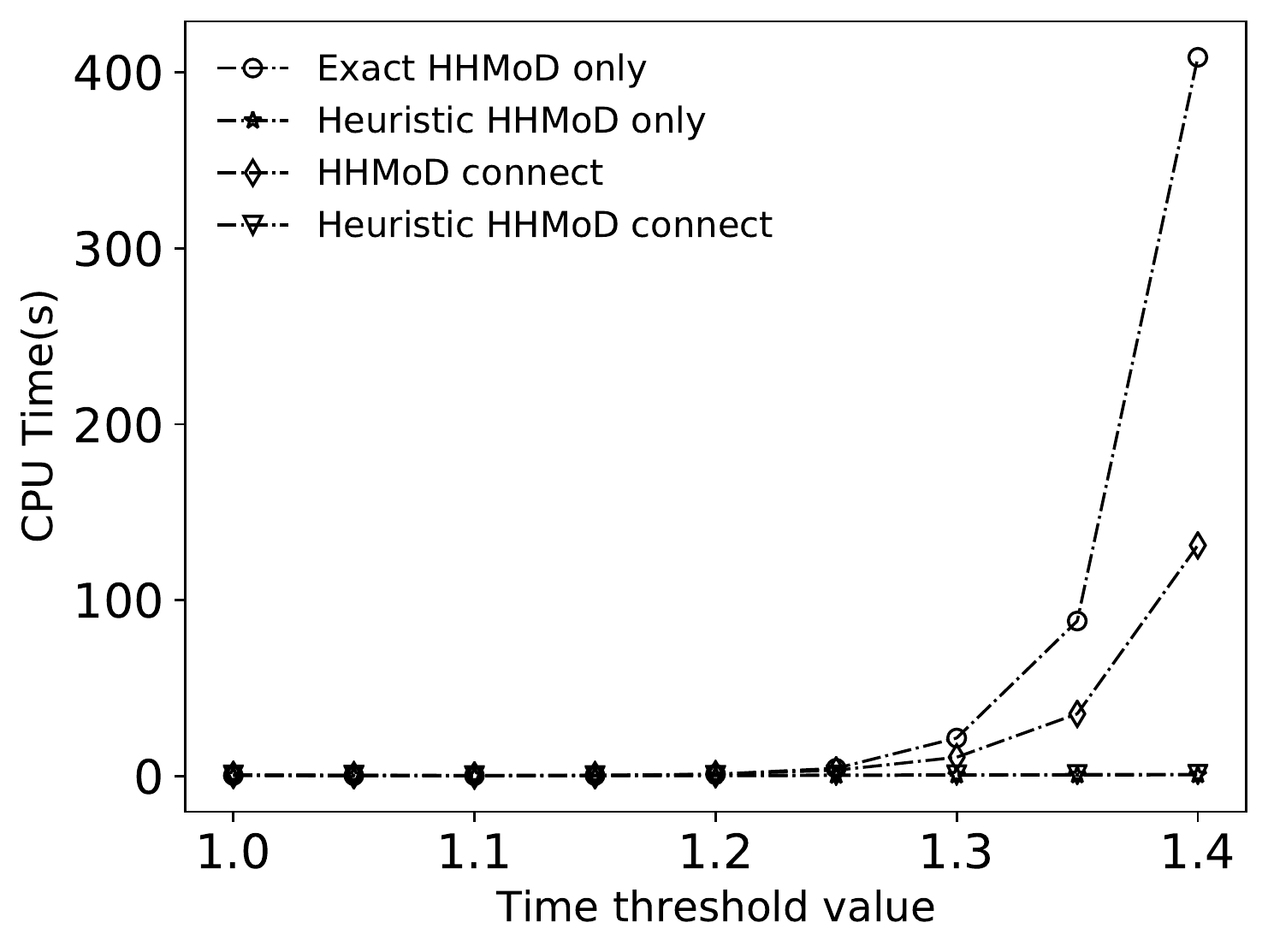}}
    \subfloat[Impact of inconvenience cost for PM peak scenario ($\lambda=1.3$)\label{fig:cpu_routes_incov}]{\includegraphics[width=0.33\linewidth]{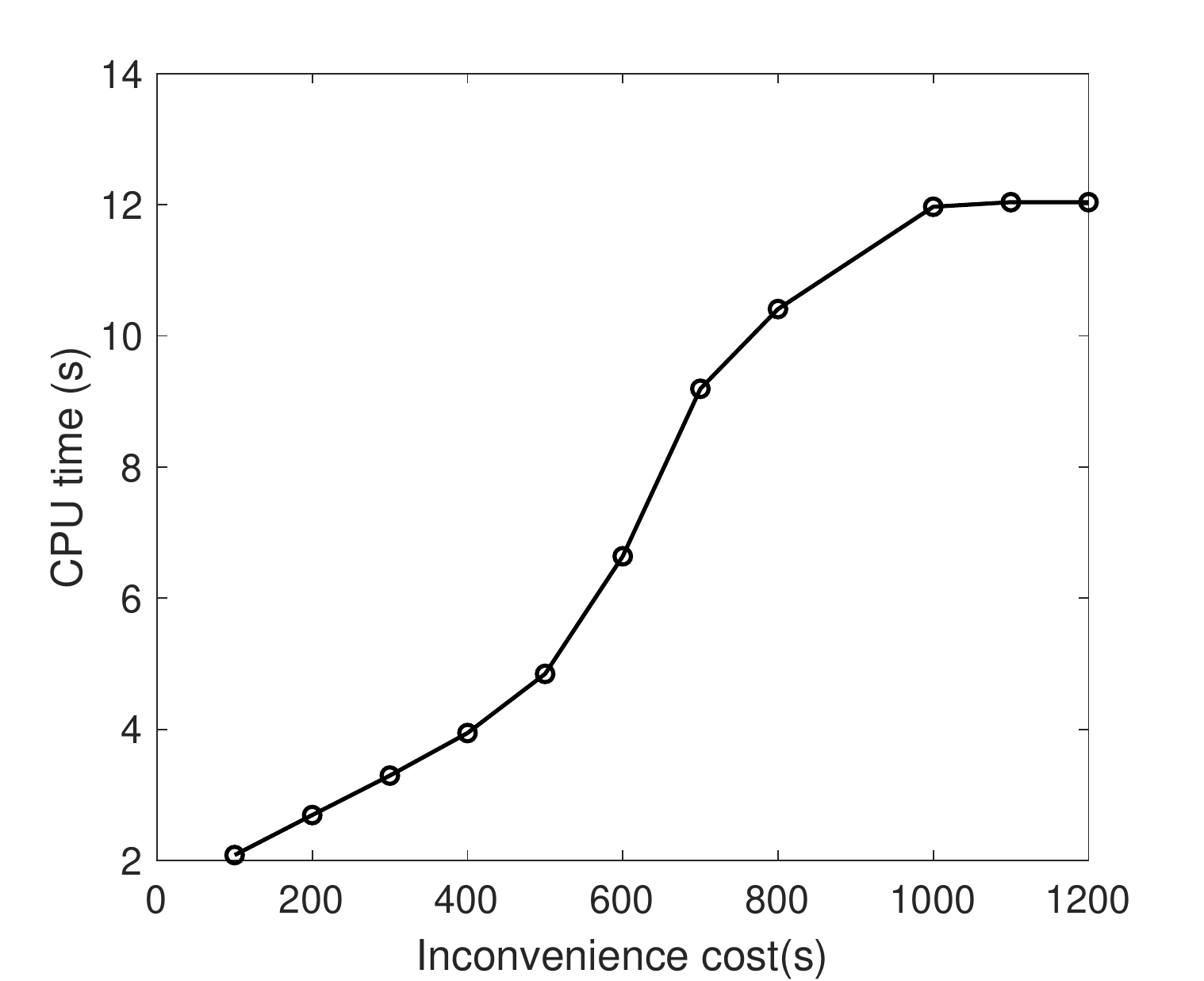}}
    \caption{Comparison of computation performances of the four different algorithms with varying time threshold values}
    \label{fig:cpu_routes}
\end{figure}

We next present the effectiveness of the route coverage algorithms for the four different times of the day and the results are summarized in Figure~\ref{fig:coverage_routes_compare}. The figure visualizes the cumulative passenger coverage rate with an increasing number of identified routes. We find that all exact and heuristic approaches in our study have superior performances compared to the shortest path based heuristic method. While the shortest path based heuristic may achieve similar passenger coverage for the first few generated routes, there is a significant gap in passenger coverage with additional routes generated. Specifically, the shortest path heuristic plateau at the passenger coverage of around 70\% and is unable to reach 100\% passenger coverage even with a large number of routes generated. In contrast, the exact algorithms can reach 100\% passenger coverage with around 32 candidate routes, whereas the heuristic algorithm can also cover all the stops with 38 to 40 routes. And all exact and heuristic algorithms are found to deliver consistent performances with varying passenger demand profiles at different times of the day.

\begin{figure}[h!]
    \centering
    \subfloat[AM peak]{\includegraphics[width=0.33\linewidth]{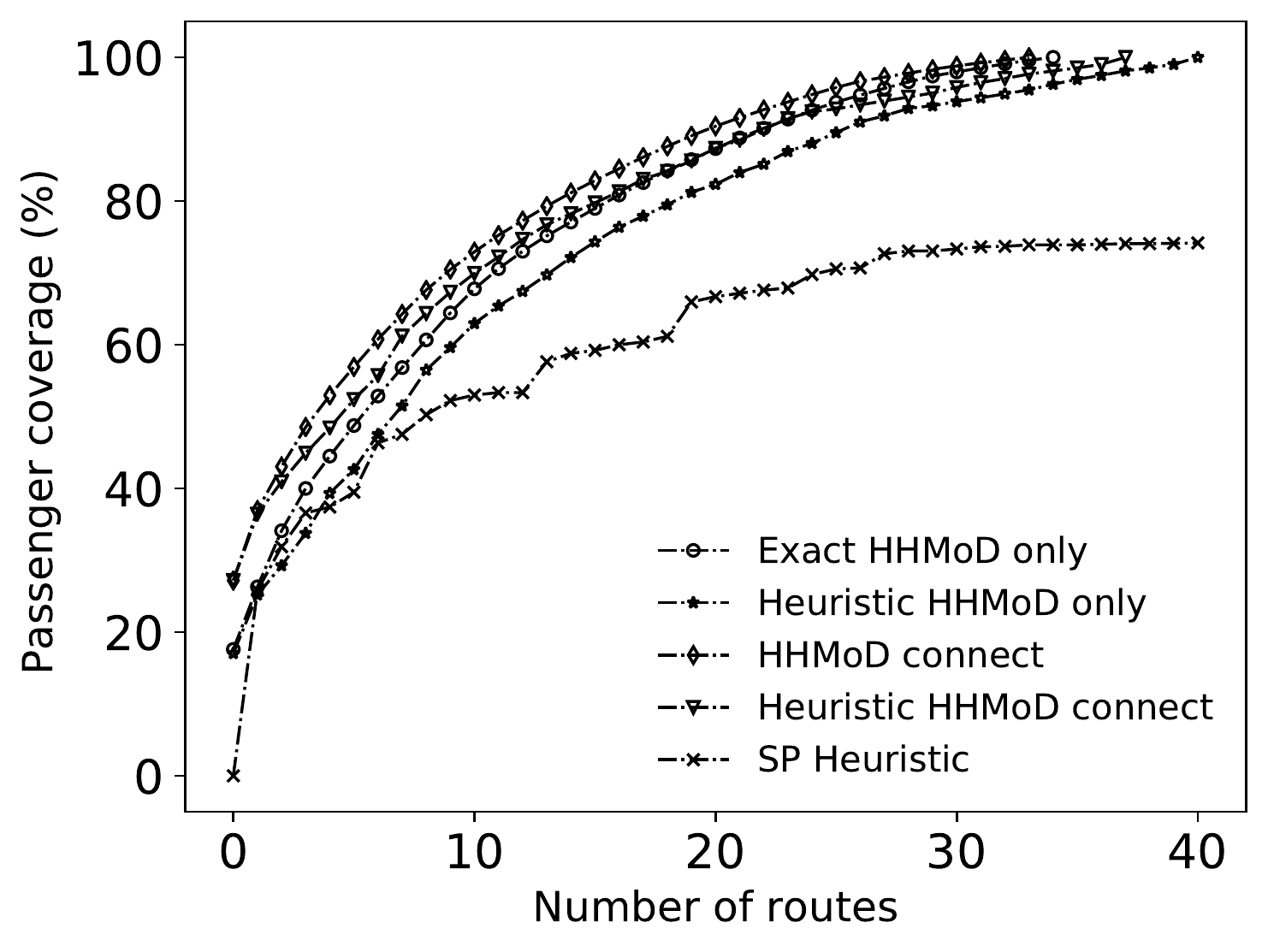}}
    \subfloat[off-peak]{\includegraphics[width=0.33\linewidth]{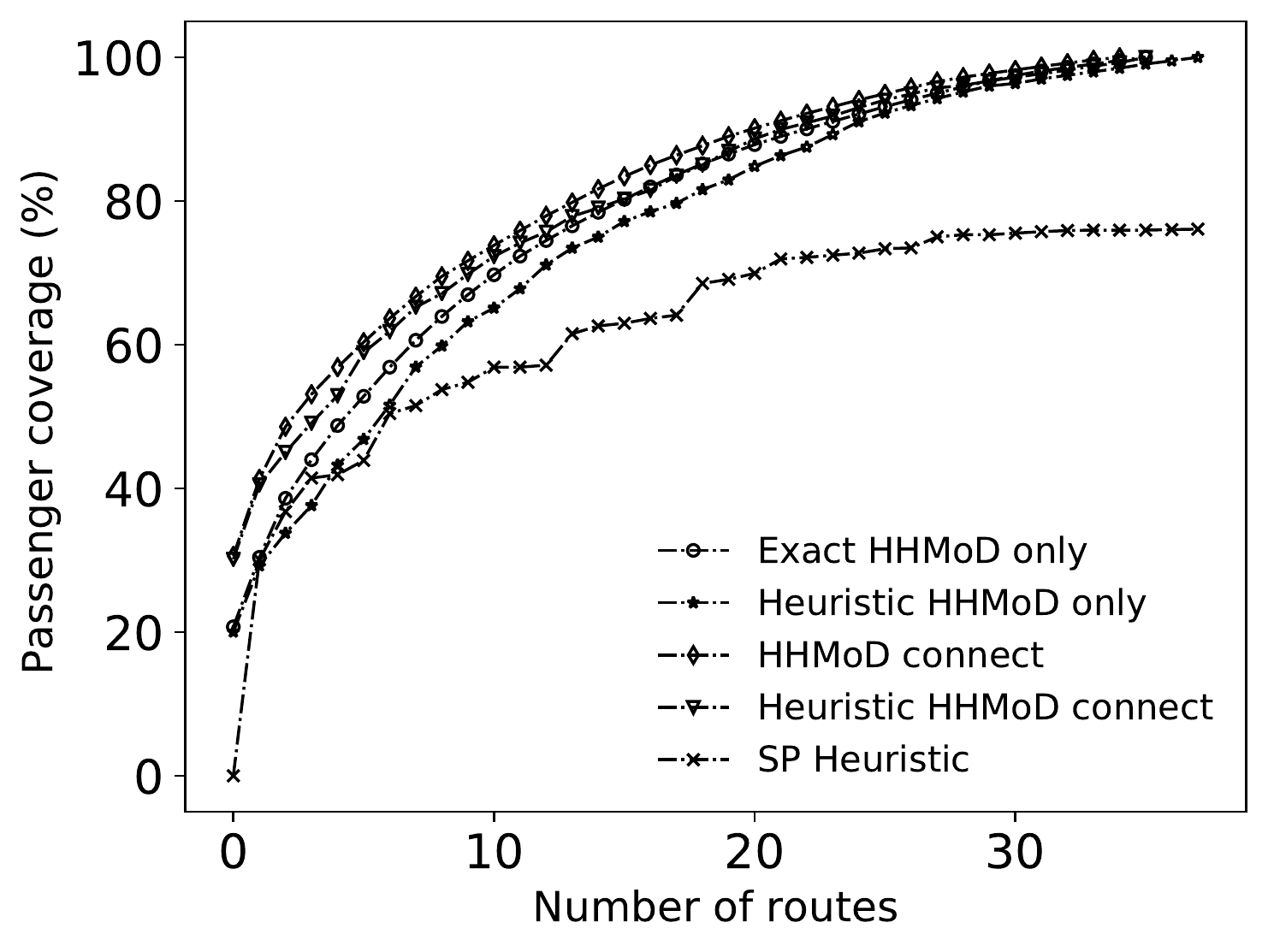}}\\
    \subfloat[PM peak]{\includegraphics[width=0.33\linewidth]{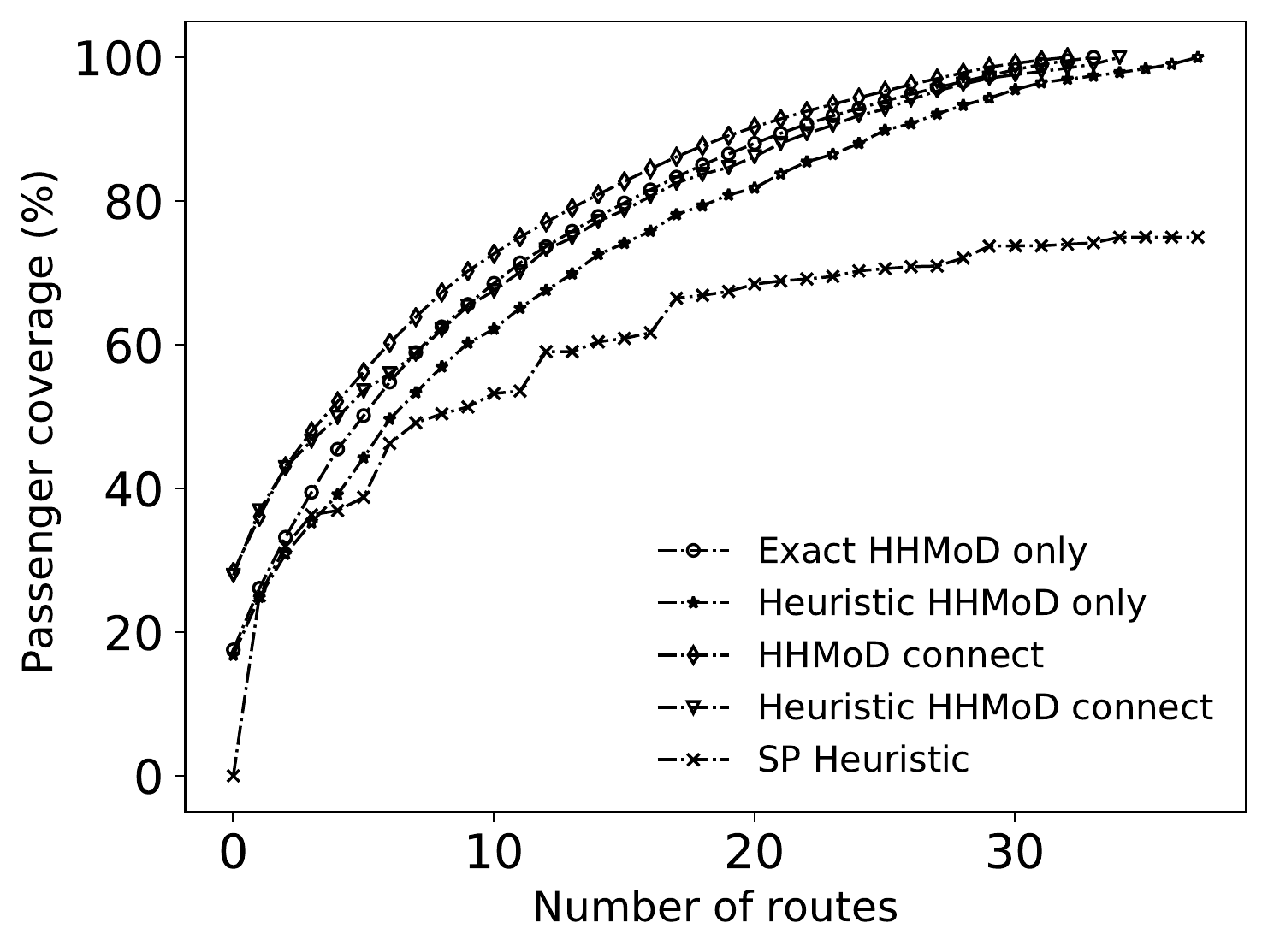}}
    \subfloat[Night time]{\includegraphics[width=0.33\linewidth]{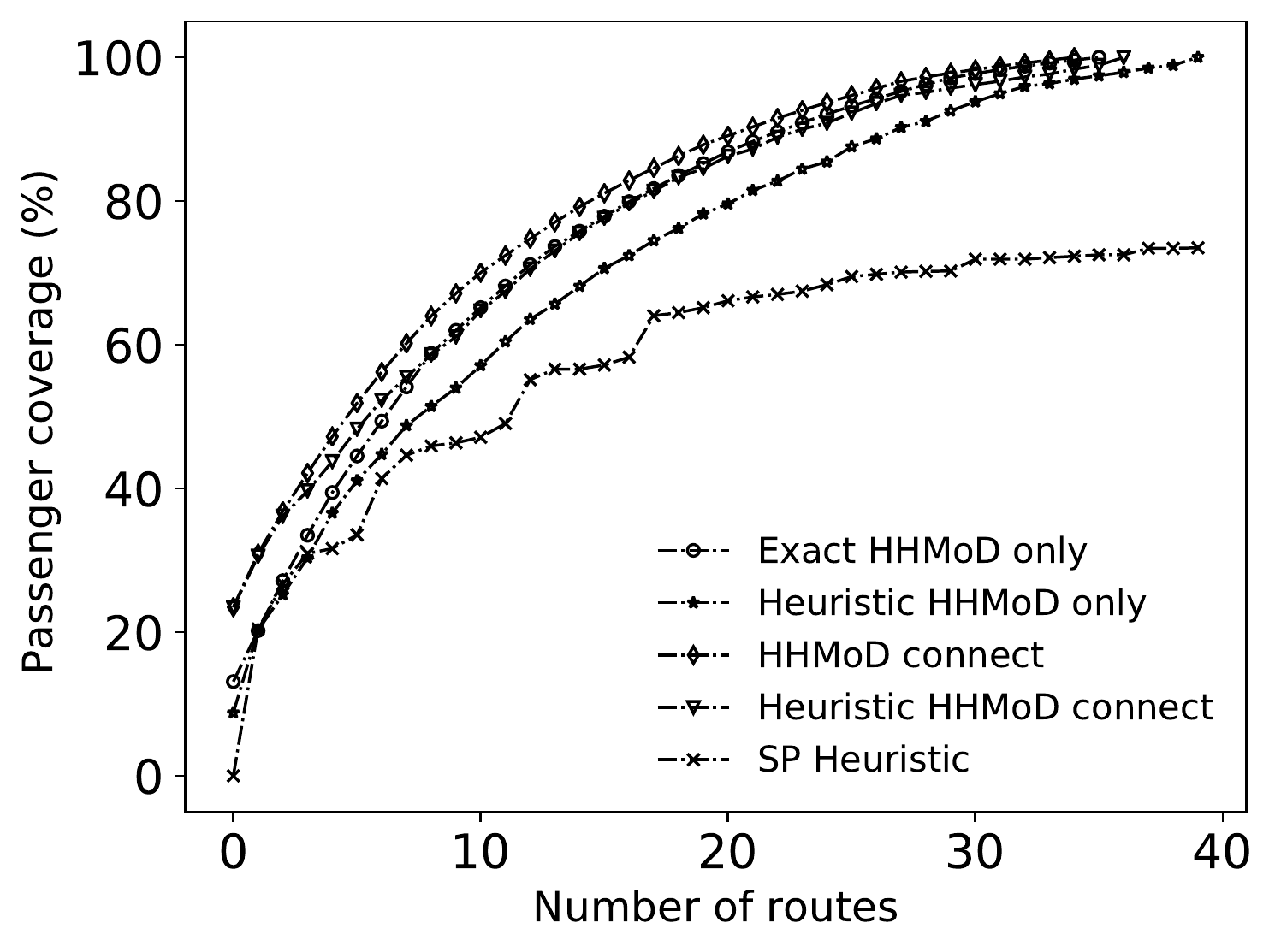}}
    \caption{The passenger coverage level (from hub) with respect to the number of routes generated, with $\lambda=1.3$.}
    \label{fig:coverage_routes_compare}
\end{figure}
\begin{figure}[h!]
    \centering
    \subfloat[AM peak - 5 routes coverage]{\includegraphics[width=0.33\linewidth]{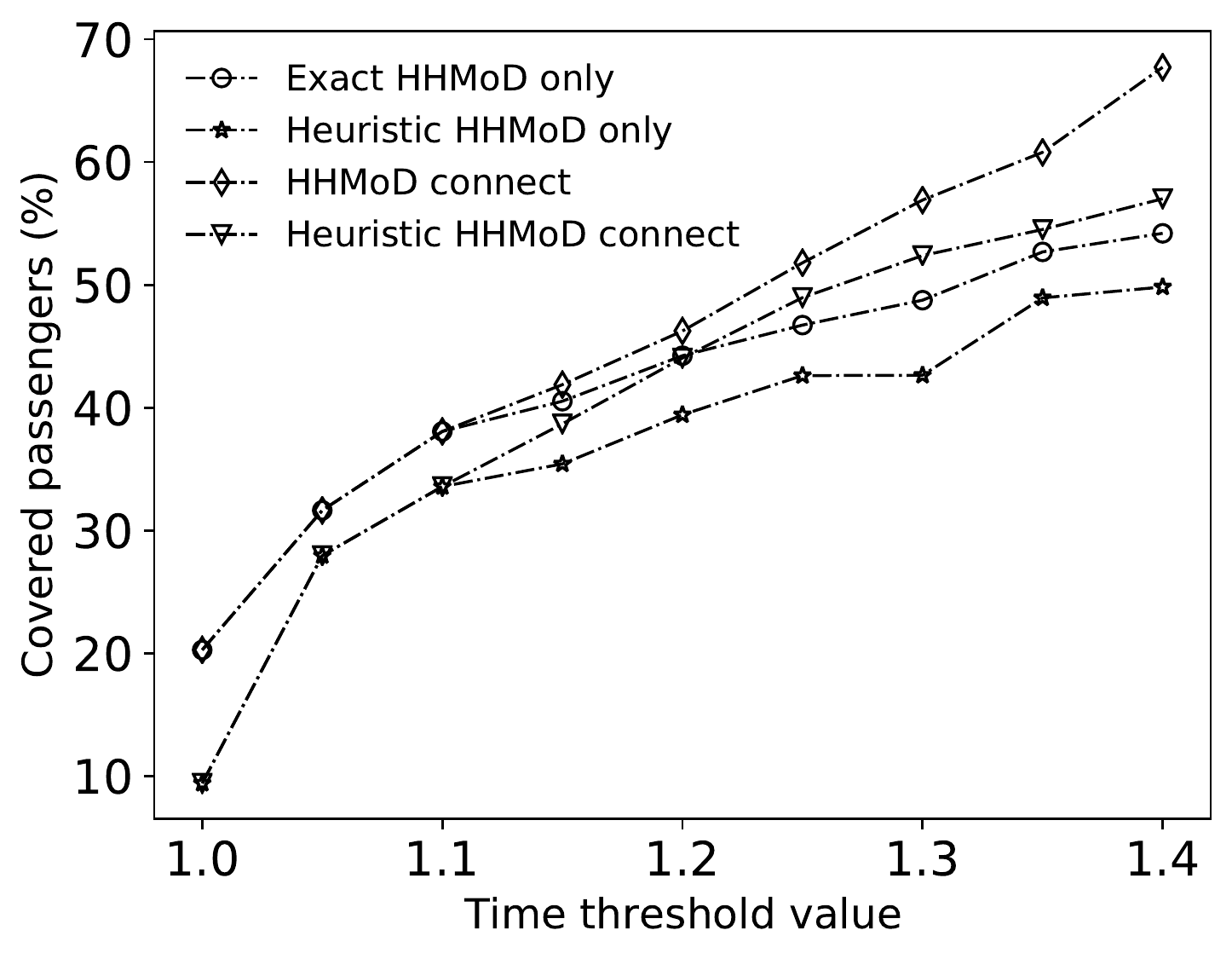}}
    \subfloat[PM peak]{\includegraphics[width=0.33\linewidth]{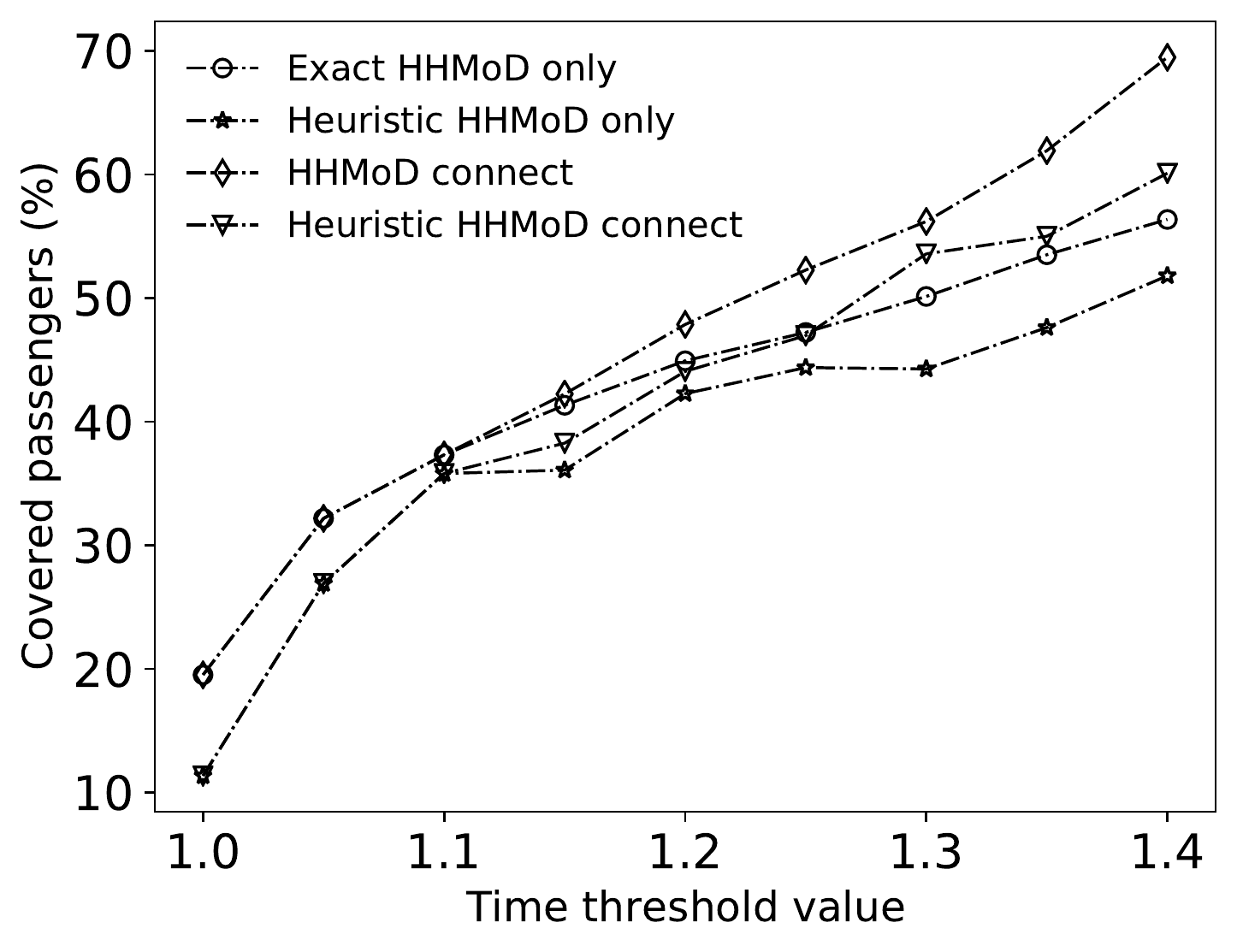}}\\
    \subfloat[AM peak - number of routes to cover]{\includegraphics[width=0.33\linewidth]{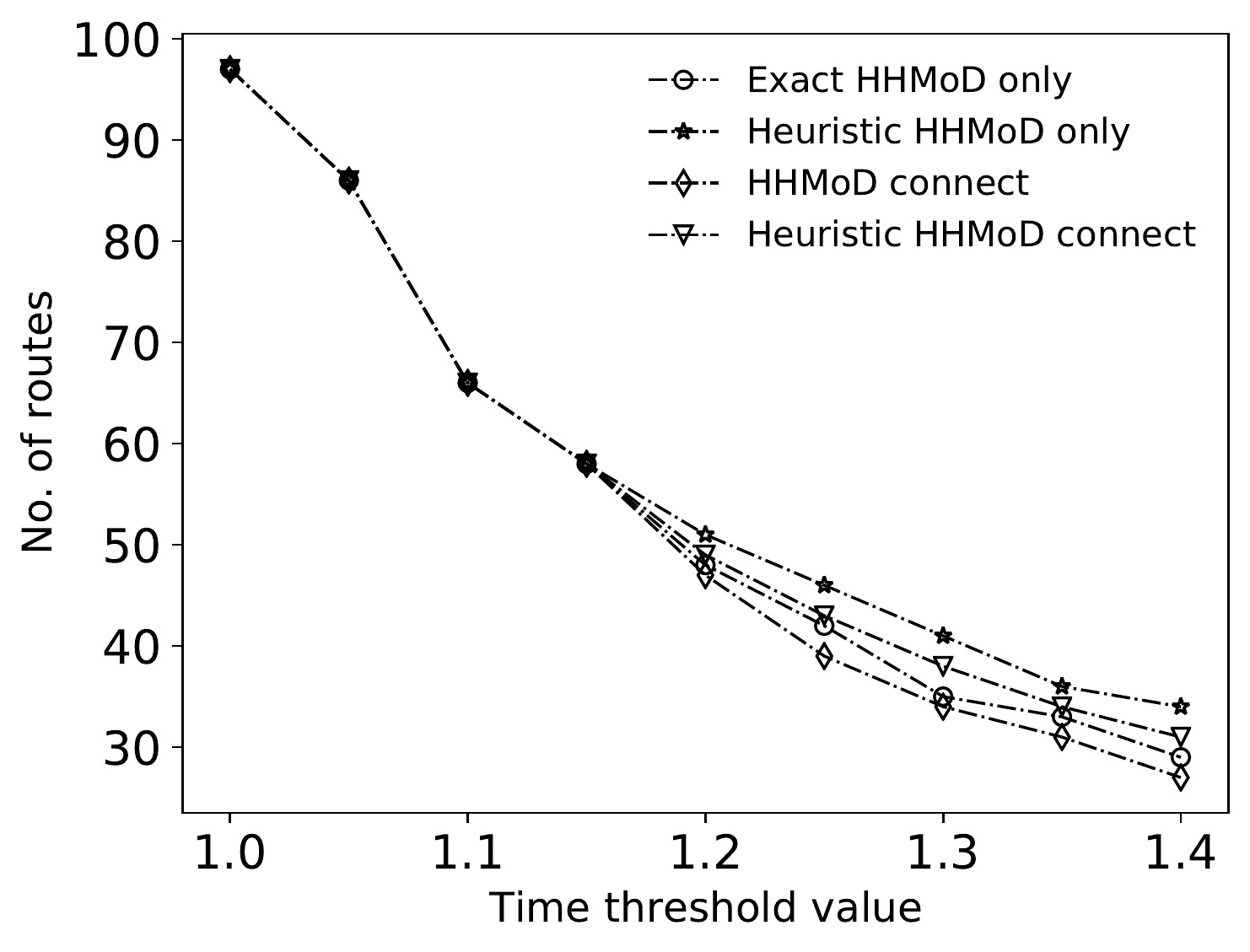}}
    \subfloat[PM peak]{\includegraphics[width=0.33\linewidth]{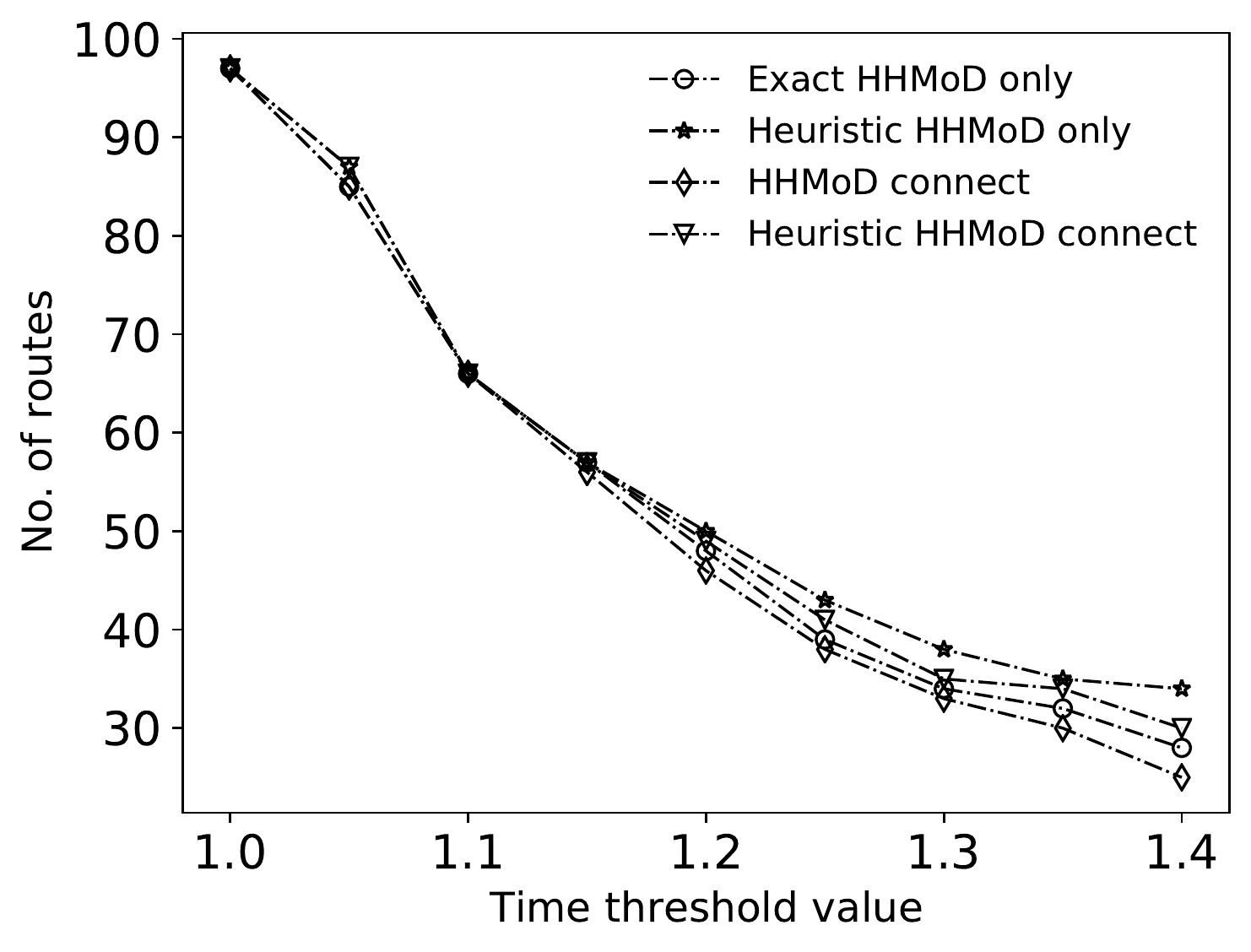}}\\
    \caption{Passenger coverage with top 5 candidate routes}
    \label{fig:coverage_routes_topK_Nroute}
\end{figure}

By varying the $\lambda$ value,  we can further examine the quality of the generated routes based on (1) the demand coverage level of the top candidate routes and (2) the number of routes required to cover all passenger demand. The first indicator is important for decision-makers who seek to make the best use of their limited vehicle resources and the second indicator suggests how decision-makers may maximize the service coverage with minimal cost. Here we only demonstrate the results for the routes generated for passengers departed from the hubs during AM and PM peak periods, and similar performances are observed for other scenarios and for passengers traveling to the hubs. Figures~\ref{fig:coverage_routes_topK_Nroute} (a)-(b) show the total passenger coverage of the top 5 routes with a time threshold value between 1 and 1.4. It can be observed that the exact algorithms for HHMoD only and HHMoD connect have identical performances with a small time threshold value (e.g., $\lambda\leq 1.15$). In these cases, the additional travel time is not sufficient to compensate for the transfer and walking time to and from the subway system and none of the generated paths is connected to the NYC subway. As the time threshold increases, the top 5 routes may cover up to 70\% of the passenger demand by connecting the HHMoD service to the subway system, where the corresponding coverage without such a connection is between 44\% to 55\% depending on the particular time of the day. The results illustrate that a sufficiently high passenger coverage level can be reached with a few routes by the proposed route generation approach, especially during day time, which only requires a reasonable relaxation of the travel time as compared to the direct service from taxis and FHVs. These results highlight the effectiveness of the HHMoD service in serving the passengers at activity hubs and also suggest notable benefits for integrating the HHMoD service with existing transportation facilities. This observation at JFK supports our initial speculation where the HHMoD can be a promising ridesharing solution to consolidates passenger to and from major activity hubs. In addition, we observe that the performances of the heuristic algorithms are 5\% to 10\% inferior as compared to their exact counterparts for the top 5 candidate routes and the resulting passenger coverage of the heuristic HHMoD connect approach may reach over 55\%, which surpasses the coverage of the exact HHMoD only scenario. This indicates that the heuristic algorithms can also generate satisfactory candidate routes and the performances are appealing for large-scale real-world cases considering its much lower computation costs. As for the number of required routes to cover all passenger demand, we observe that fewer than 30 routes may be sufficient with the HHMoD connect approach in all scenarios as shown in Figures~\ref{fig:coverage_routes_allpass} (c)-(d). This represents an additional 20 routes to serve the rest 30\% to 40\% of the passengers that are unsatisfied by the top 5 candidate routes. Without connecting to the subway system, the HHMoD only approach requires 3 to 8 more candidate routes to fulfill all passenger demand as compared to the HHMoD connect approach. And the heuristic methods would require an extra 2 to 10 routes depending on the particular times of the day. 

\begin{figure}[h!]
    \centering
    \subfloat[AM peak]{\includegraphics[width=0.5\linewidth]{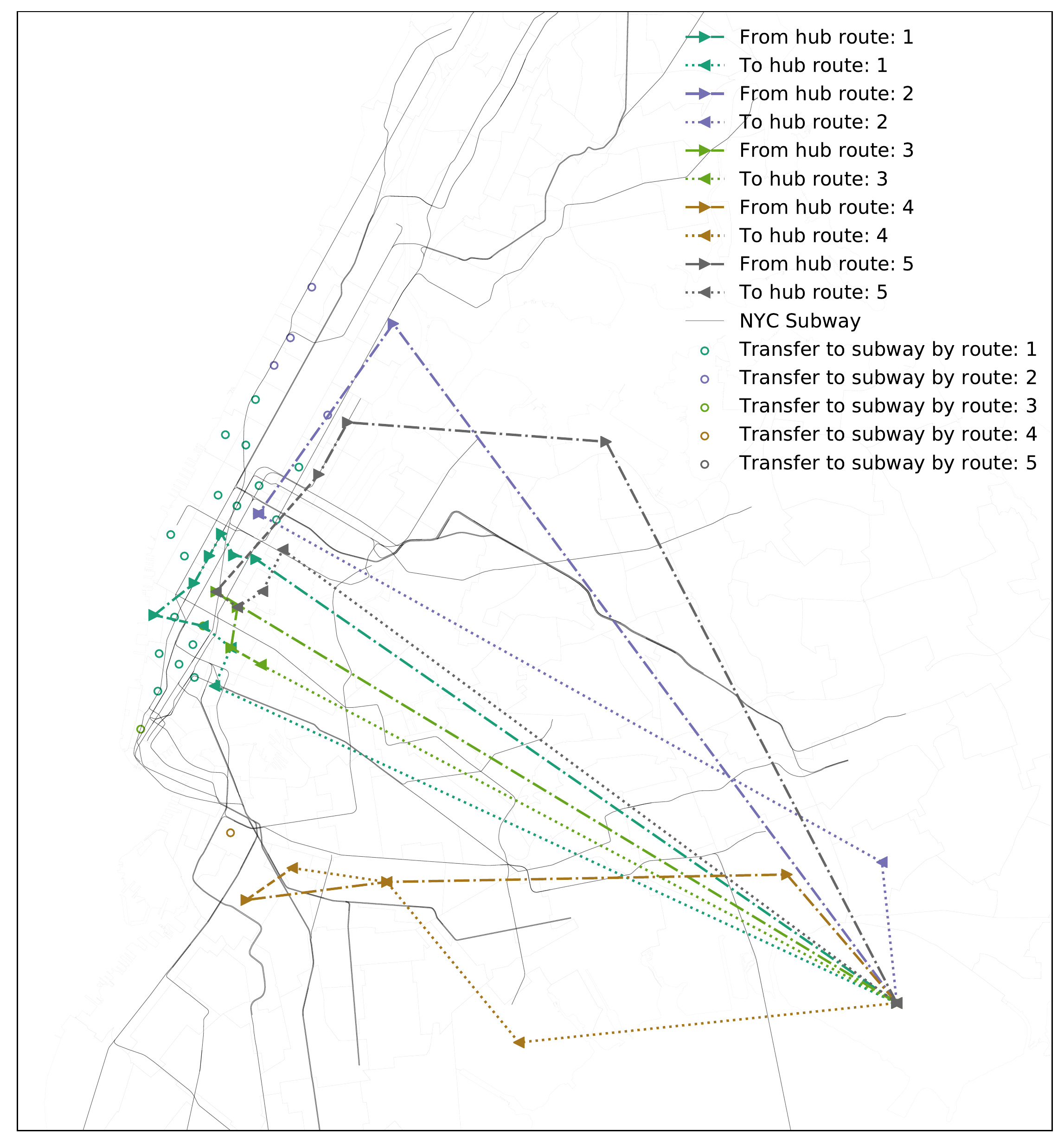}}
    \subfloat[off-peak]{\includegraphics[width=0.5\linewidth]{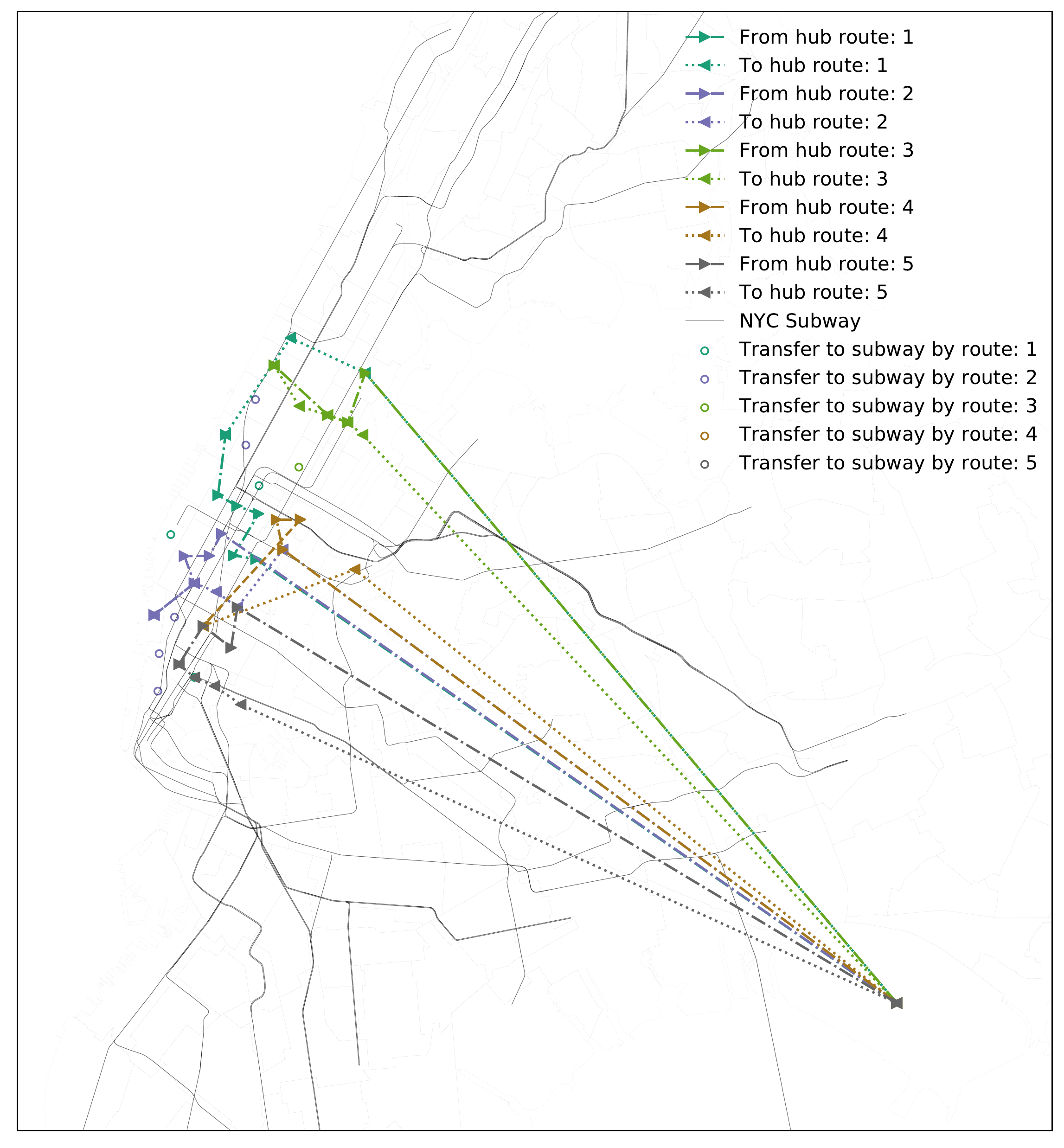}}\\
    \subfloat[PM peak]{\includegraphics[width=0.5\linewidth]{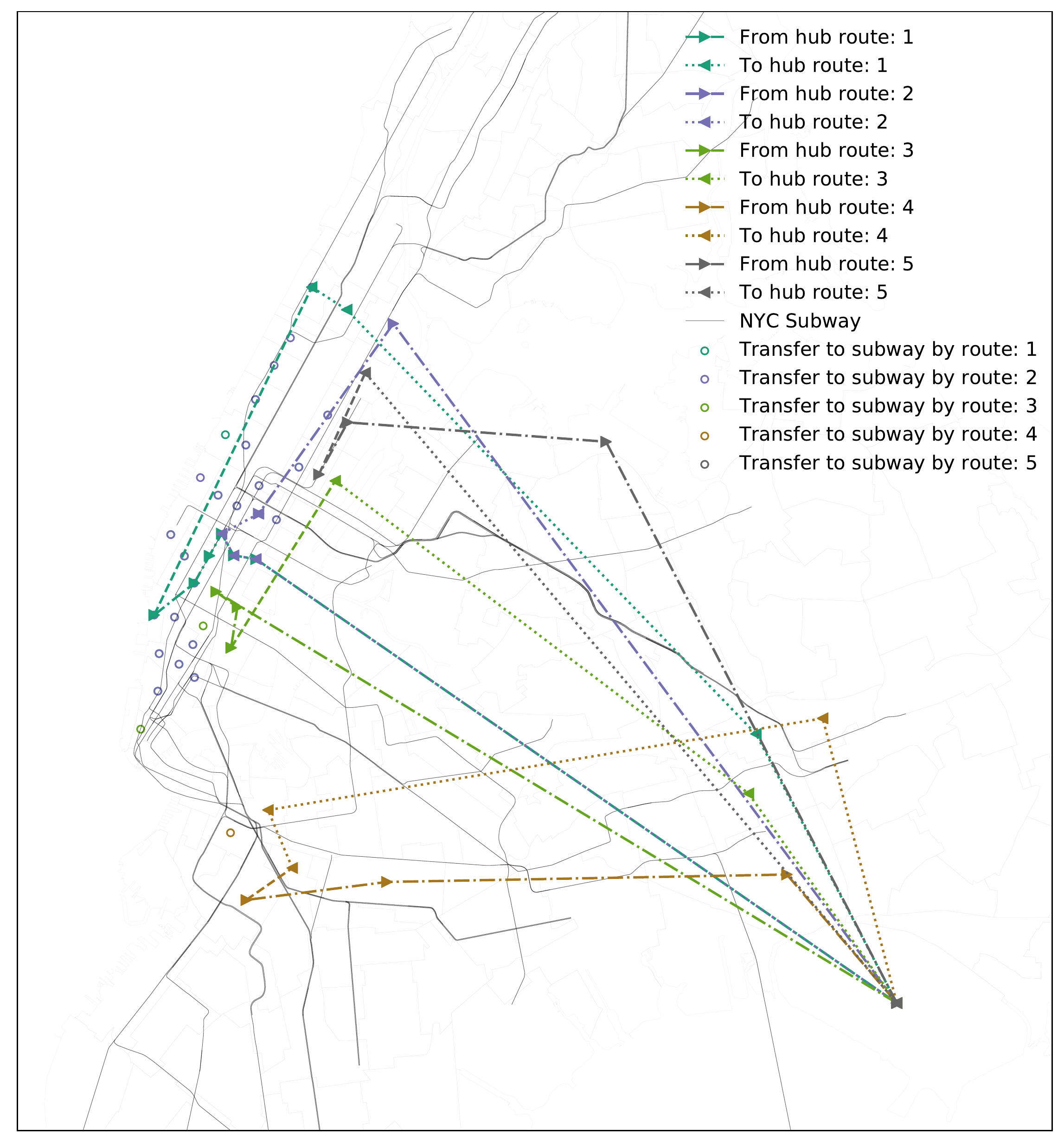}}
    \subfloat[Night time]{\includegraphics[width=0.5\linewidth]{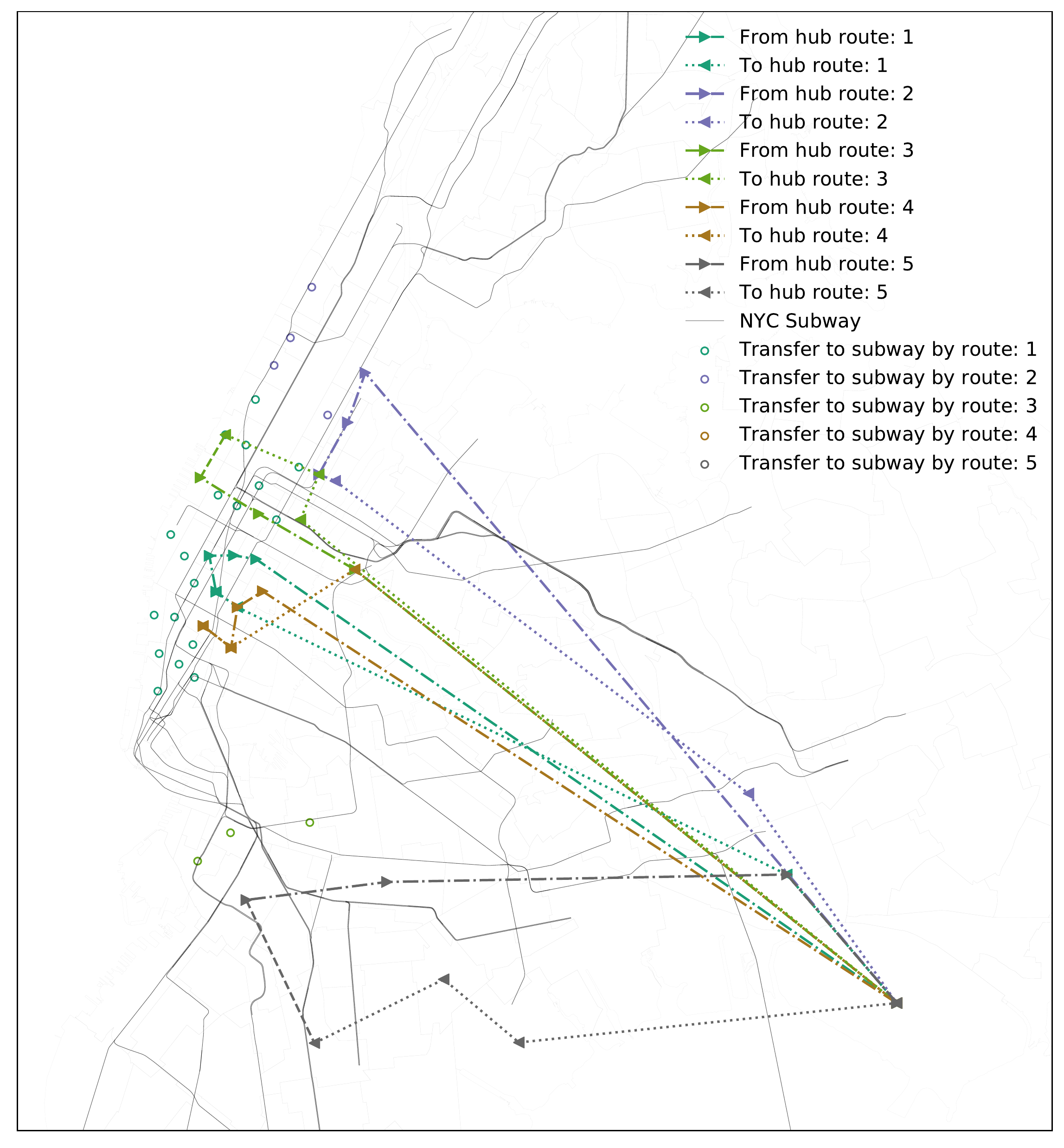}}
    \caption{Visualization of top 5 routes generated with connection to subway}
    \label{fig:viz_DATconnect}
\end{figure}
\begin{figure}[h!]
    \centering
    \subfloat[AM peak]{\includegraphics[width=0.5\linewidth]{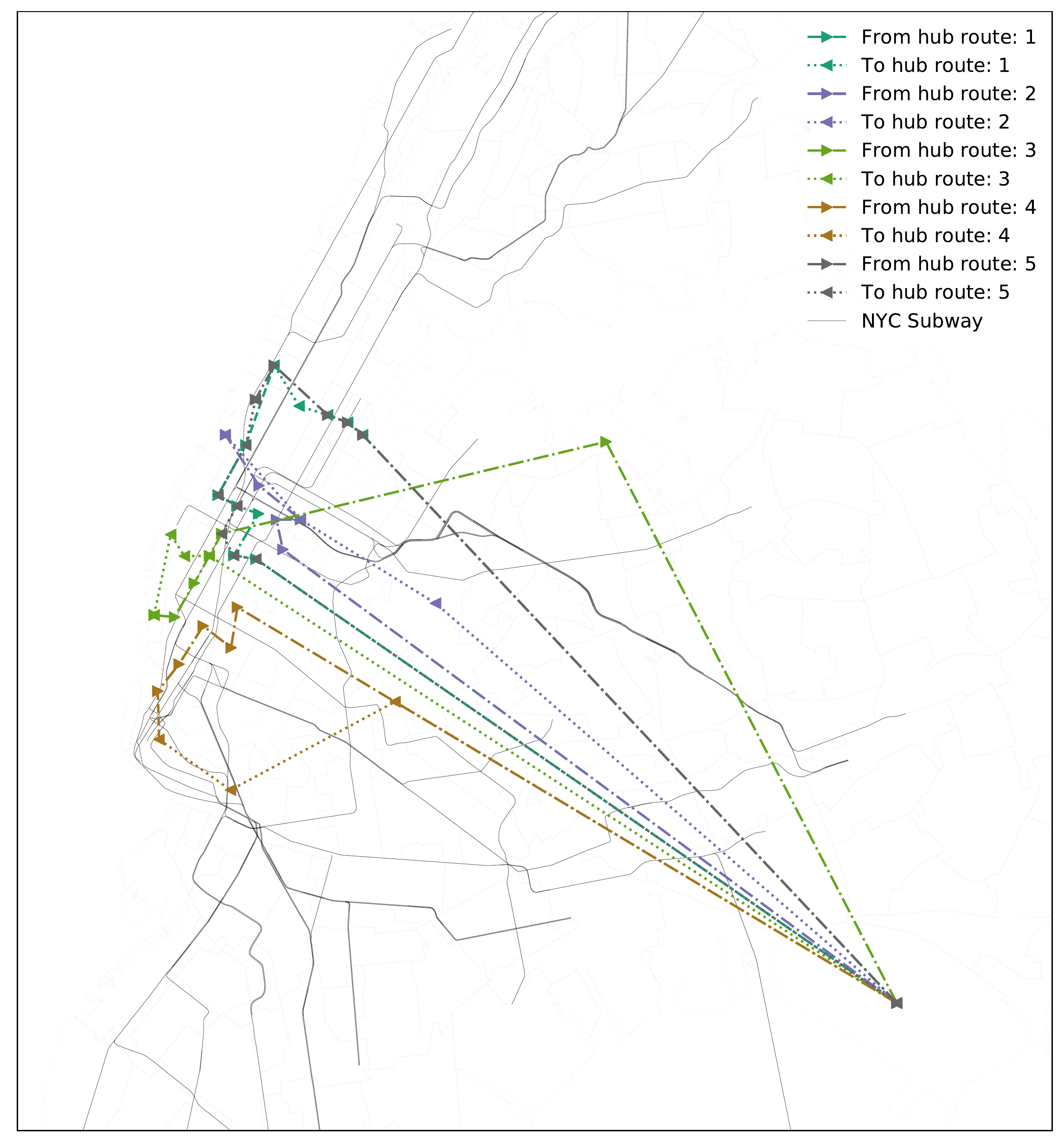}}
    \subfloat[off-peak]{\includegraphics[width=0.5\linewidth]{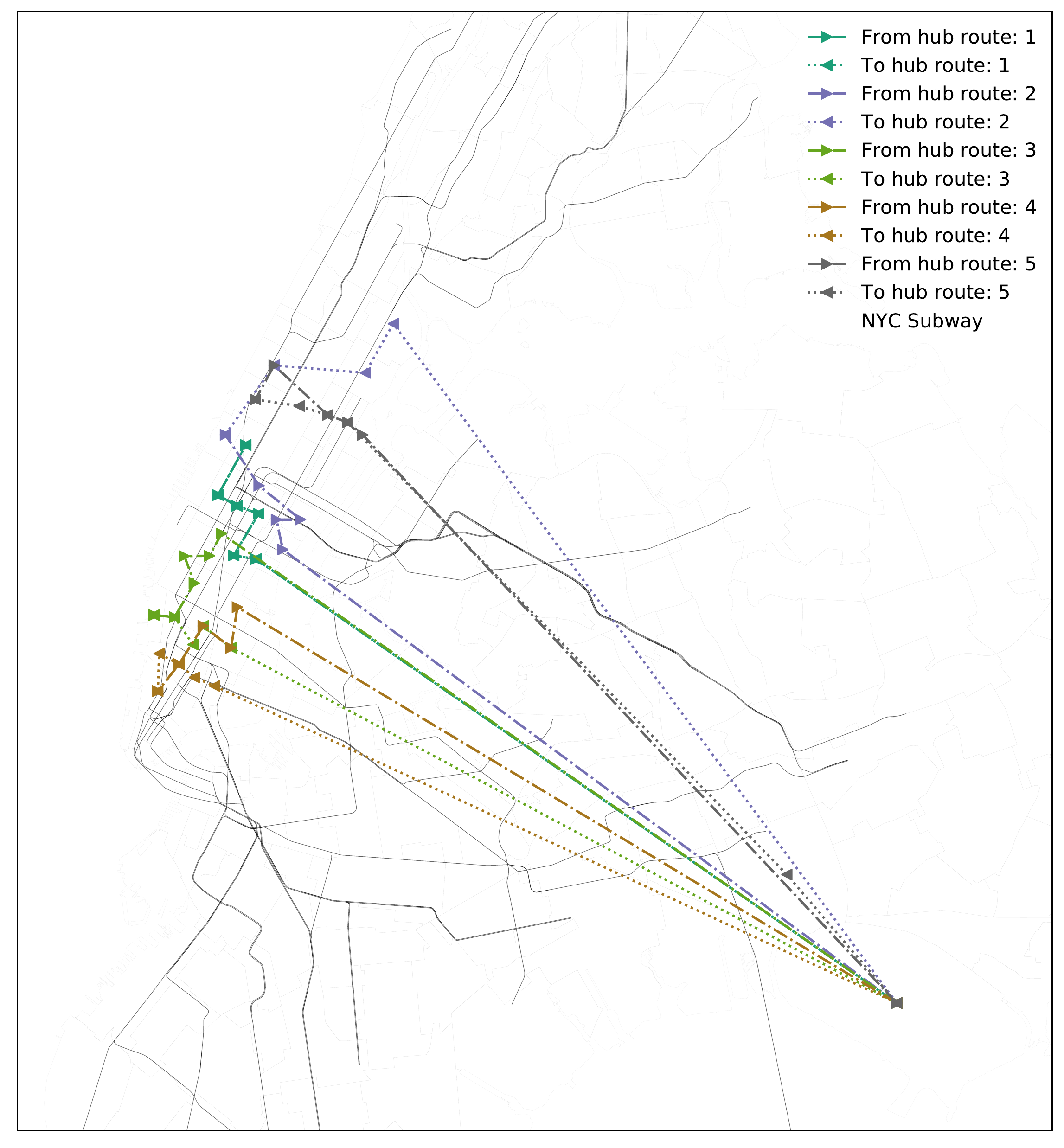}}\\
    \subfloat[PM peak]{\includegraphics[width=0.5\linewidth]{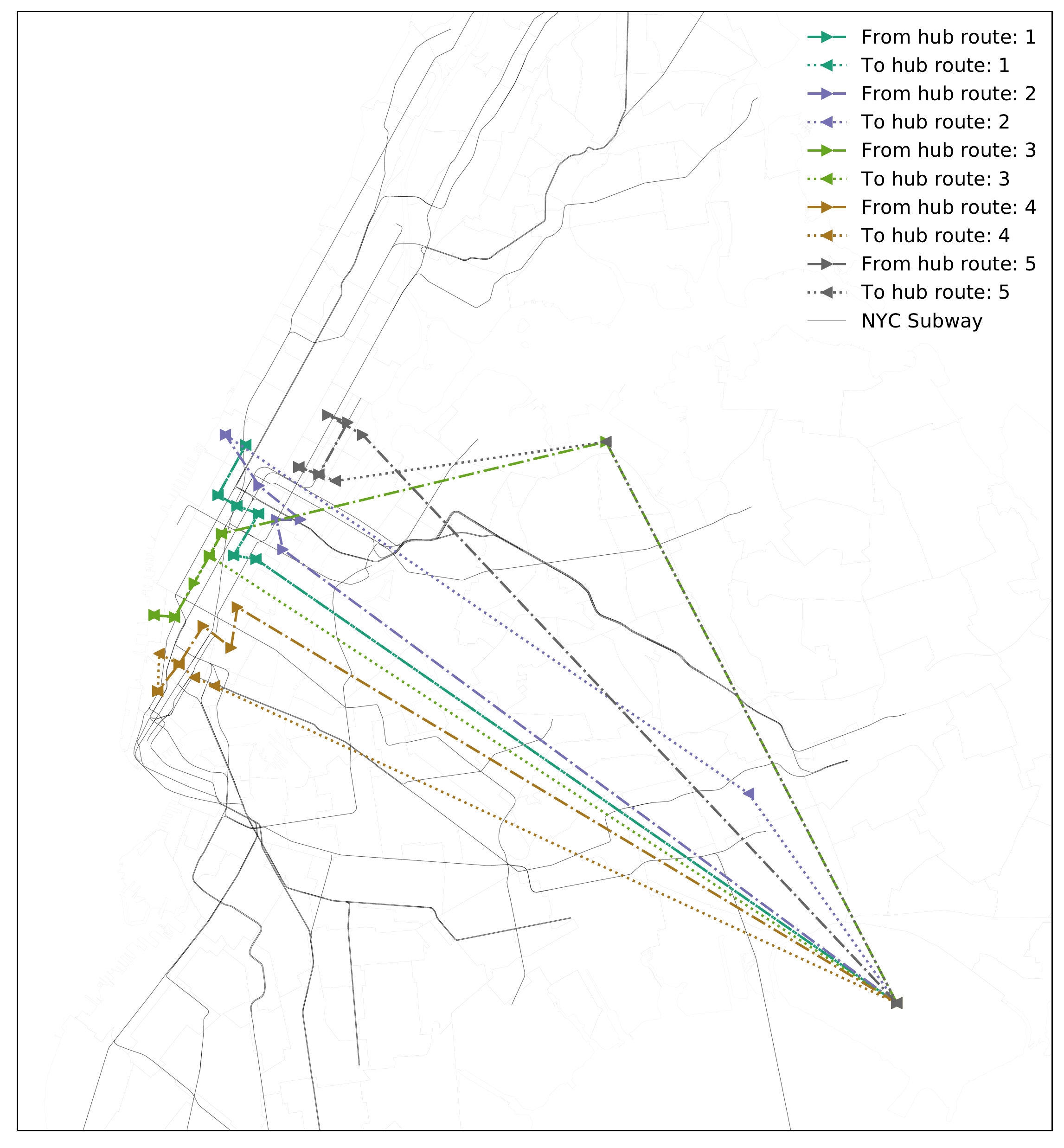}}
    \subfloat[Night time]{\includegraphics[width=0.5\linewidth]{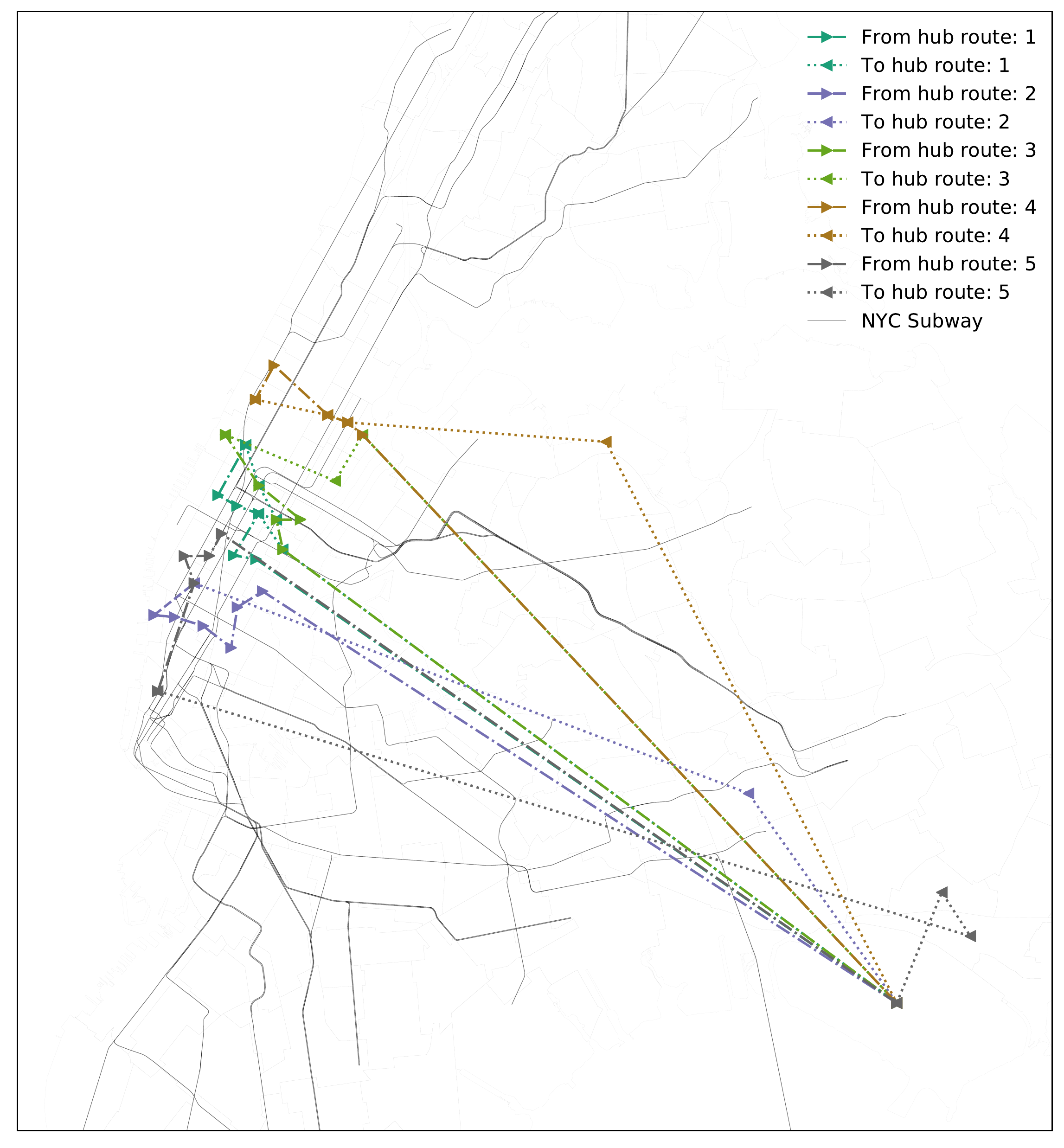}}
    \caption{Visualization of top 5 routes generated without connecting to subway}
    \label{fig:viz_DATonly}
\end{figure}

Finally, by combining the candidate routes to and from the activity hub following equation~\ref{eq:combination}, we arrive at the completed round trip routes for serving the demand to and from the hub, and we visualize the outcomes of the top 5 round trip routes in Figures~\ref{fig:viz_DATconnect} and~\ref{fig:viz_DATonly}. In the figures, we distinguish the segments to and from the hubs with different lines. We also visualize the underlying layout of the NYC subway system and mark the stops that are served by connecting to the subway system in circles. As we can see in the figures, for maximum passenger coverage, the top routes generated for both HHMoD only and HHMoD connect share similar philosophy by prioritizing stops in Manhattan of high passenger demand but also visit several intermediate locations in Queens along the route which is permissible under the travel time constraint. And for all scenarios, the incoming and outgoing routes at the JFK airport are observed to be well-paired to form the round trip loops based on their trip demand and connection distance. From the results, we can also clearly tell the differences between the HHMoD only routes and HHMoD connect routes driven by the distinct philosophies of the underlying algorithms. In particular, the top 2-3 HHMoD connect routes are aligned with the NYC subway lines in Manhattan, and the passenger coverage is maximized by connecting to passengers along multiple subway lines such as lines 1-3 and lines A-F. Unlike the HHMoD only case, the HHMoD connect routes makes fewer stops in the upper west Manhattan areas with the connections to the subway system, which saves travel time from getting stuck in the heavy traffic. This allows for additional resources to operate a dedicated HHMoD route serving the Brooklyn area during AM and PM peak hours as the other routes have already achieved sufficient coverage in Manhattan. We also note that the 3rd route during night time connects to the G line at 21st Station in Queens, making it possible to serve transfer passengers whose destinations are in the west part of Brooklyn. Finally, for both HHMoD only and HHMoD connect, the direct connecting segment between the JFK airport and the LGA airport is observed to satisfy the passenger demand that travels between the two airport hubs during AM and PM peak hours.

\subsection{Robust route scheduling}
With the K-MCR finalized, we next discuss the results for the robust route scheduling problem under demand uncertainty, and here we only focus on the scheduling of the HHMoD connect routes due to its superior performances in the route generation step. As the input for the robust scheduling problems, the candidate HHMOD connect routes are generated with $\lambda=1.3$ that cover all the candidate stops, and we consider that there are 200 available homogeneous vehicles with the capacity of 20 seats (such as medium-sized shuttle vans with rooms for luggage). We use Gurobi 9.0 package to solve the MILP for both master and recourse problems on a PC with 3.6 GHz CPU and 32GB RAM. The stopping criterion is set as the relative gap between the lower bound and the upper bound of the C\&CG algorithm being smaller than $10^{-4}$:
\begin{equation*}
\frac{UB-LB}{UB}\leq 10^{-4}    
\end{equation*}
\begin{figure}[h!]
    \centering
    \subfloat[\label{fig:converge}Example of algorithm convergence for the HHMoD connect during PM peak hours]{\includegraphics[width=0.7\linewidth]{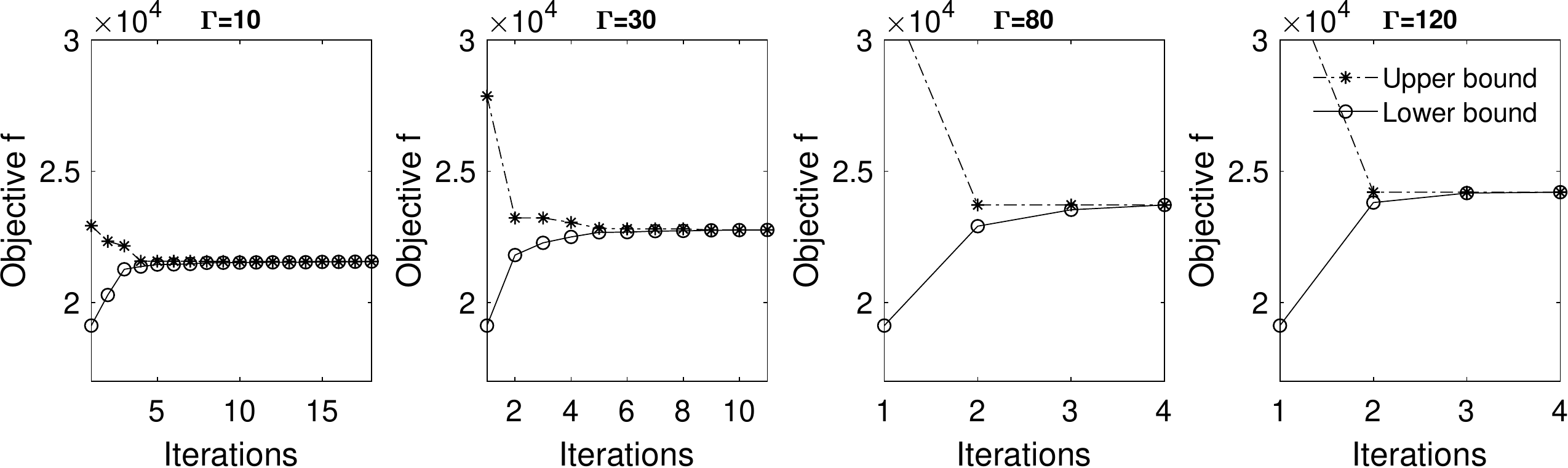}}\\
    \subfloat[CPU Time for different time of day]{\includegraphics[width=0.7\linewidth]{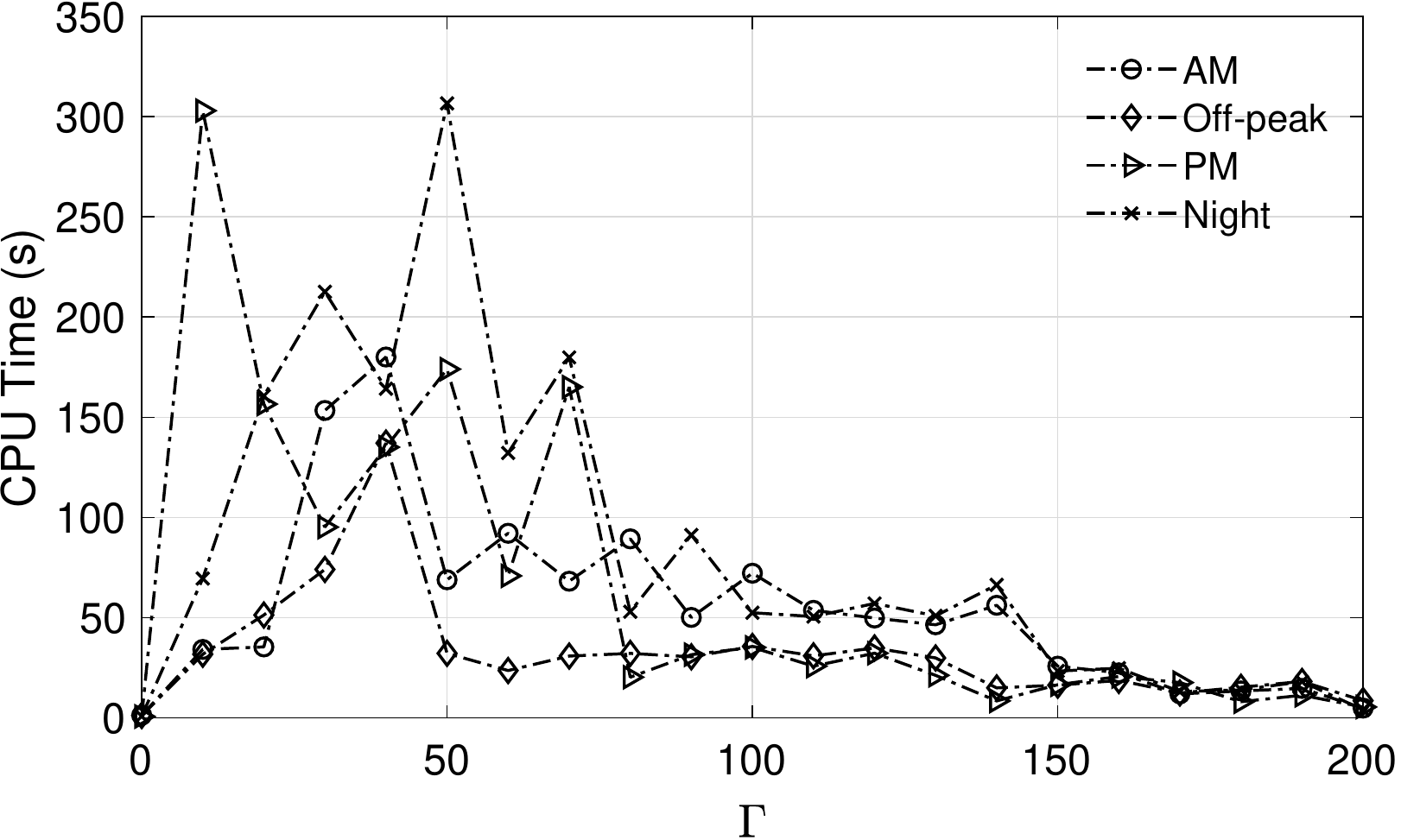}}
    \caption{Convergence and computation time for the robust management of the planned HHMoD service}
    \label{fig:cpu_robust}
\end{figure}
We first summarize the computational performances for the robust route scheduling problem from both the convergence and the computational time perspectives, and the results are shown in Figure~\ref{fig:cpu_robust}. We observe that the two-stage RO converges to the global optimum in 5-6 iterations in most cases. The convergence speed is faster for cases with very small $\Gamma$ (close to 0) or large $\Gamma$ (close to 200) values due to simpler problem structures that are similar to the nominal optimization problem or the single state optimization problem with box uncertainty constraints. In most of the scenarios, the robust route scheduling can be solved to global optimum within 200 seconds. The most computationally expensive scenarios are for $\Gamma$ values in the range between 10 to 70, where we see a higher number of iterations and longer computational time, and a snapshot of the convergence can be found in Figure~\ref{fig:converge} for the off-peak period. We recall that each additional iteration for the C\&CG algorithm will introduce a large number of variables and constraints into the master problem. While the recourse problems are relatively easier to solve, the additional iterations and constraints lead to excessive computation time for the master problem, which accounts for almost 100\% of total computation time in our experiments. On the other hand, even though certain $\Gamma$ values are found to be associated with more iterations and high computational time, the C\&CG converges quickly in the first 8 iterations (such as the cases with $\Gamma=10,30$ for PM peak scenario). Therefore, one may choose to perform early stopping and terminate with a minor optimality gap for reduced computational time with nearly identical performances under demand uncertainty. As for other $\Gamma$ values, a small optimality gap can be achieved in only 3-5 iterations. In this regard, although large-scale two-stage RO is in general difficult to solve, our numerical experiments suggest that the robust route scheduling problem can still be solved within a reasonable time frame. And this supports the applicability of the RO framework for the scheduling of urban HHMoD services over a short-time period. 

\begin{figure}[H]
    \centering
    \subfloat[Change of objective function values\label{fig:rb_total_all}]{\includegraphics[width=0.4\linewidth]{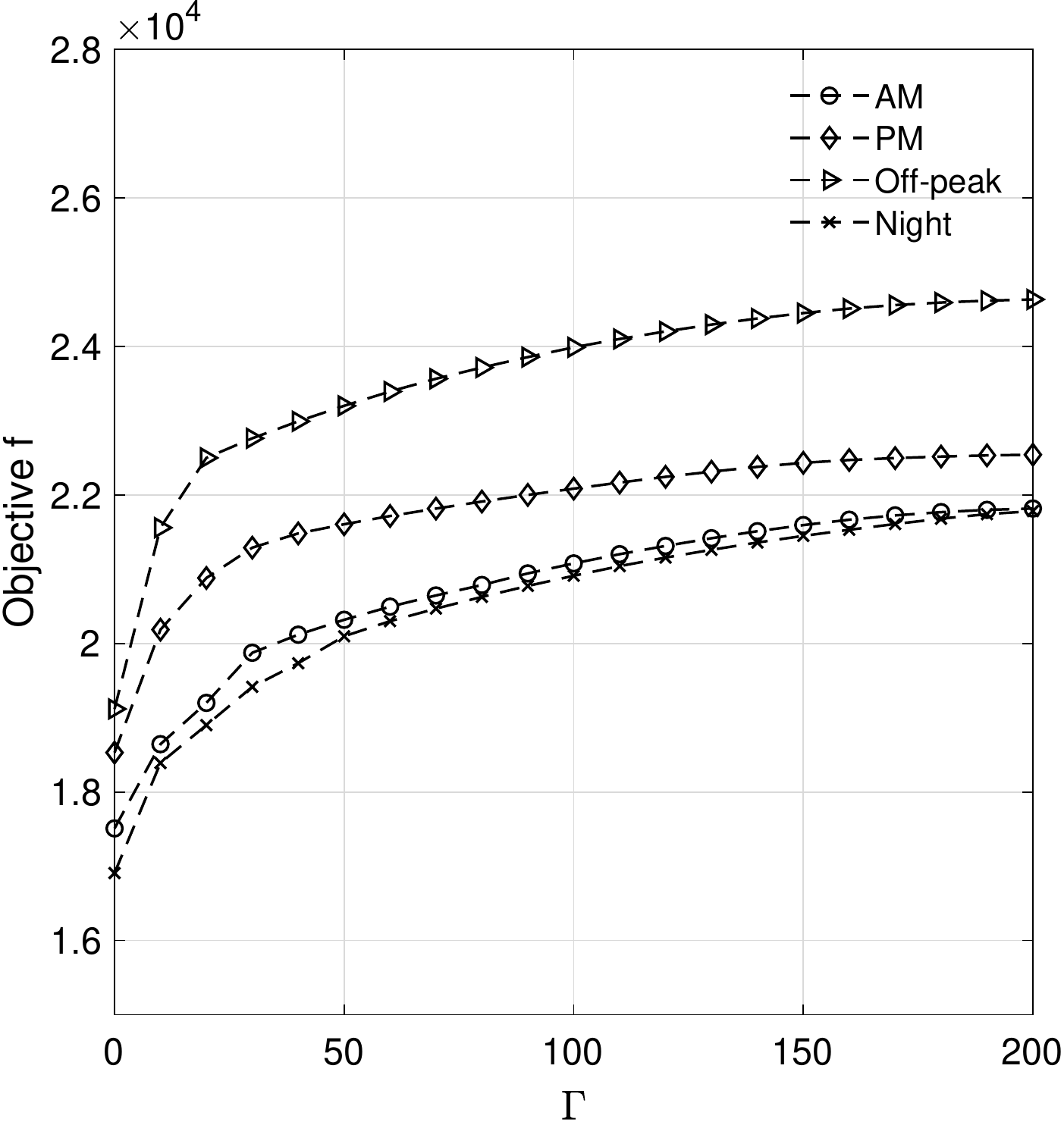}}
    \subfloat[Decompose objective function into $F_{operation}$ and $F_{waiting}+F_{loss}$\label{fig:fun_decom}]{\includegraphics[width=0.5\linewidth]{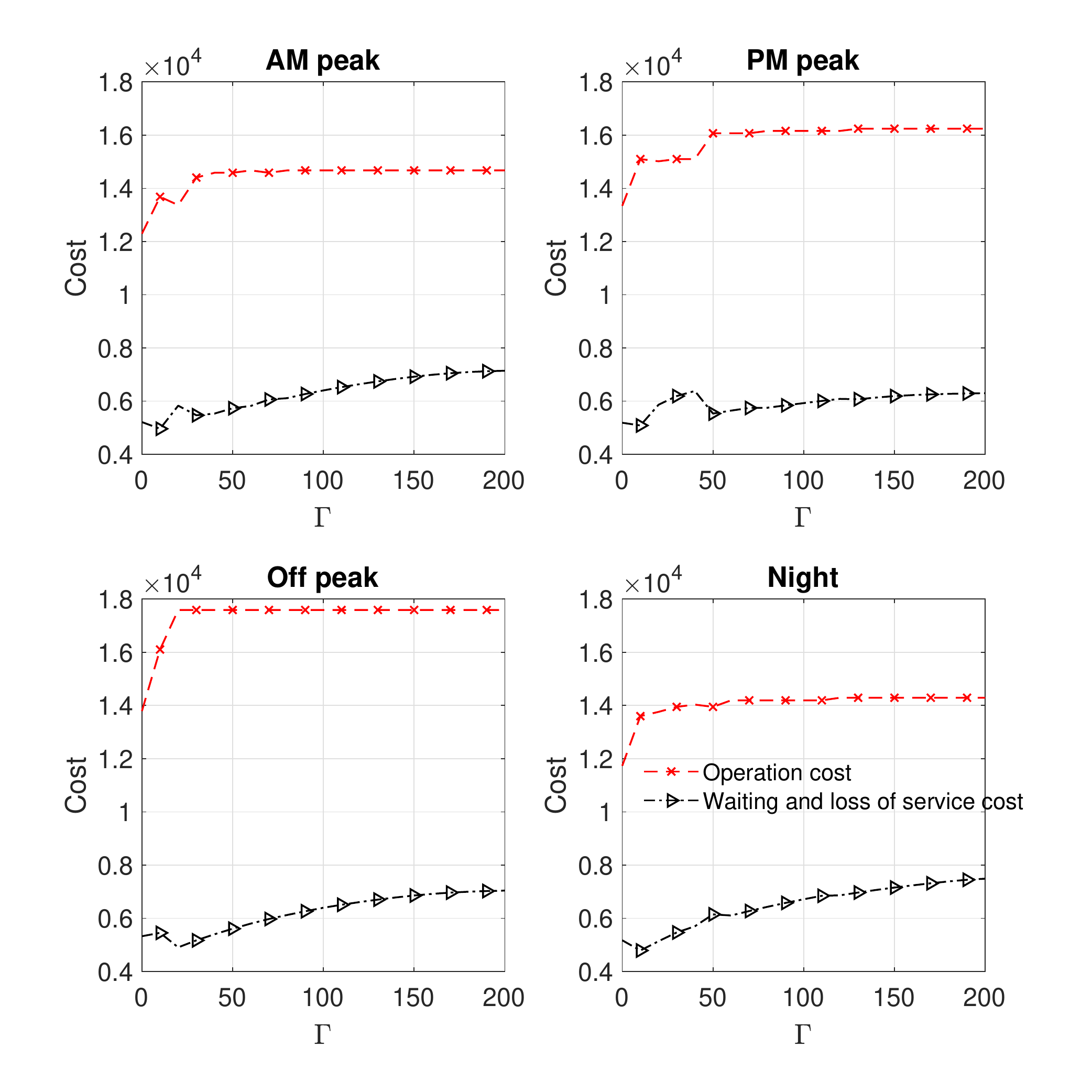}}
    \caption{Change of objective function with respect to different demand uncertainty}
\end{figure}

Next, we discuss the impacts of the budget uncertainty on the performances of the scheduled HHMoD routes, and Figure~\ref{fig:rb_total_all} demonstrates how the objective function may change with increasing budget level. We observe that the objective function value resembles a concave and monotonically increasing function as the level of demand uncertainty increases. This observation is consistent across different times of the day. Based on the concave function shape, we report that majority of the impacts due to the demand uncertainty can be accounted for by modeling the demand uncertainty with a budget of 25\% of the stops ($\Gamma=50$) with a sharp increase in objective function values as compared to the nominal function value ($\Gamma=0$). And the objective function values increase slowly for including the remaining stops in the budget of uncertainty. In addition, the changes in function values also reveal the differences in the robust route scheduling at a different time of the day. Specifically, the off-peak period is found to be associated with the highest cost, followed by the PM peak, AM peak, and night period. The reason is due to that the passenger demand level for PM peak and off-peak period at JFK is higher than that of AM peak and nigh time period. In addition, the off-peak period has the highest cost since the destinations, and origins of passengers to and from the hub are more diverse, and hence requires operating more vehicles (as shown in Figure~\ref{fig:fun_decom}) to avoid the large penalty for the loss of services. And more conservative route schedules can be observed for the off-peak and night periods as the gap of function values between $\Gamma=200$ and $\Gamma=0$ is larger than the other two cases. And such conservative schedules are necessary as the demand distribution tends to have higher variations for the off-peak and night time periods.  

In the end, we present the cost analysis to justify the values of the robust routing schedules by comparing the changes in objective function values under average scenarios and under heavy demand variations. We generate the same 100 random demand realizations for the average scenarios and calculate the expected cost for the robust route schedules with different $\Gamma$ values. The cost-effectiveness in the average case can be measured as:

\begin{equation}
    \bar{G}_{\Gamma,0}=\frac{\bar{F}_{\Gamma}-\bar{F}_0}{\bar{F}_0}
\end{equation}
where $\bar{F}_{\Gamma}$ represents the average objective function values for the planned route schedules with budget $\Gamma$. We also measure the worst-case performances of the robust schedules by considering two demand variation scenarios: (1) the full variation case with all $p_i=1$ and (2) the half variation case with all $p_i=0.5$. And the corresponding cost-effectiveness under worst-case scenario can be computed as:
\begin{equation}
    \hat{G}_{\Gamma,0}=\frac{\hat{F}_{\Gamma}-\hat{F}_0}{\hat{F}_0}
\end{equation}
Here, $\hat{F}_{\Gamma}$ denotes the objective function value under demand variation. The value of $\bar{G}_{\Gamma,0}$ indicates the level of conservatism of the robust solution and the value of $\hat{G}_{\Gamma,0}$ measures the level of robustness of the route schedules as compared to the nominal results. Besides these three metrics, we also include the loss services rate as the percentage of unsatisfied demand over the total realized demand. We present the results on the gap under average and worst cases in Figure~\ref{fig:gap_cost} and summarize the loss of services and the number of operated vehicles in each cases in Table~\ref{tab:cost_analysis}.

\begin{figure}[H]
    \centering
    \subfloat[AM peak]{\includegraphics[width=0.25\linewidth]{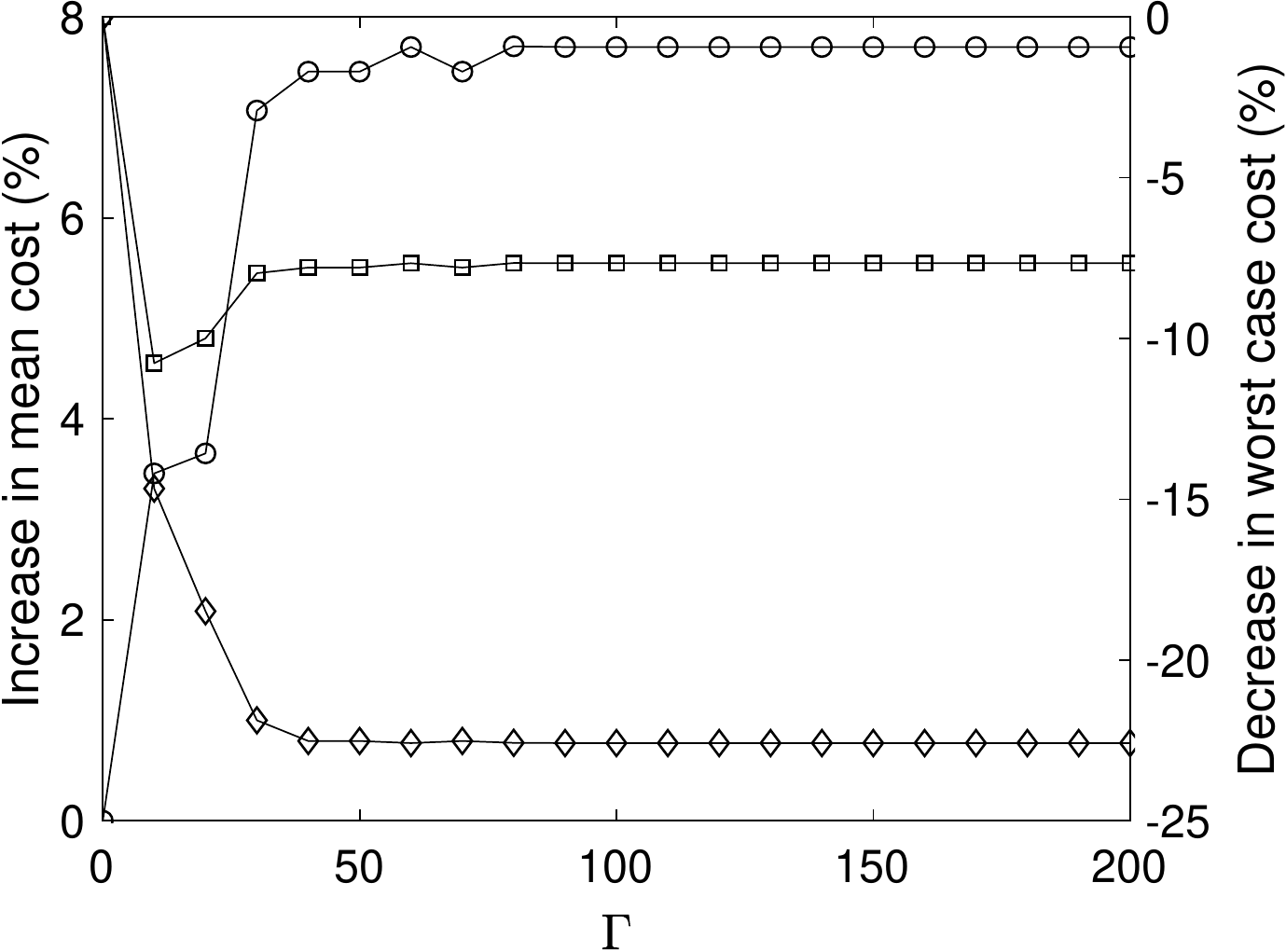}}
    \subfloat[PM peak]{\includegraphics[width=0.25\linewidth]{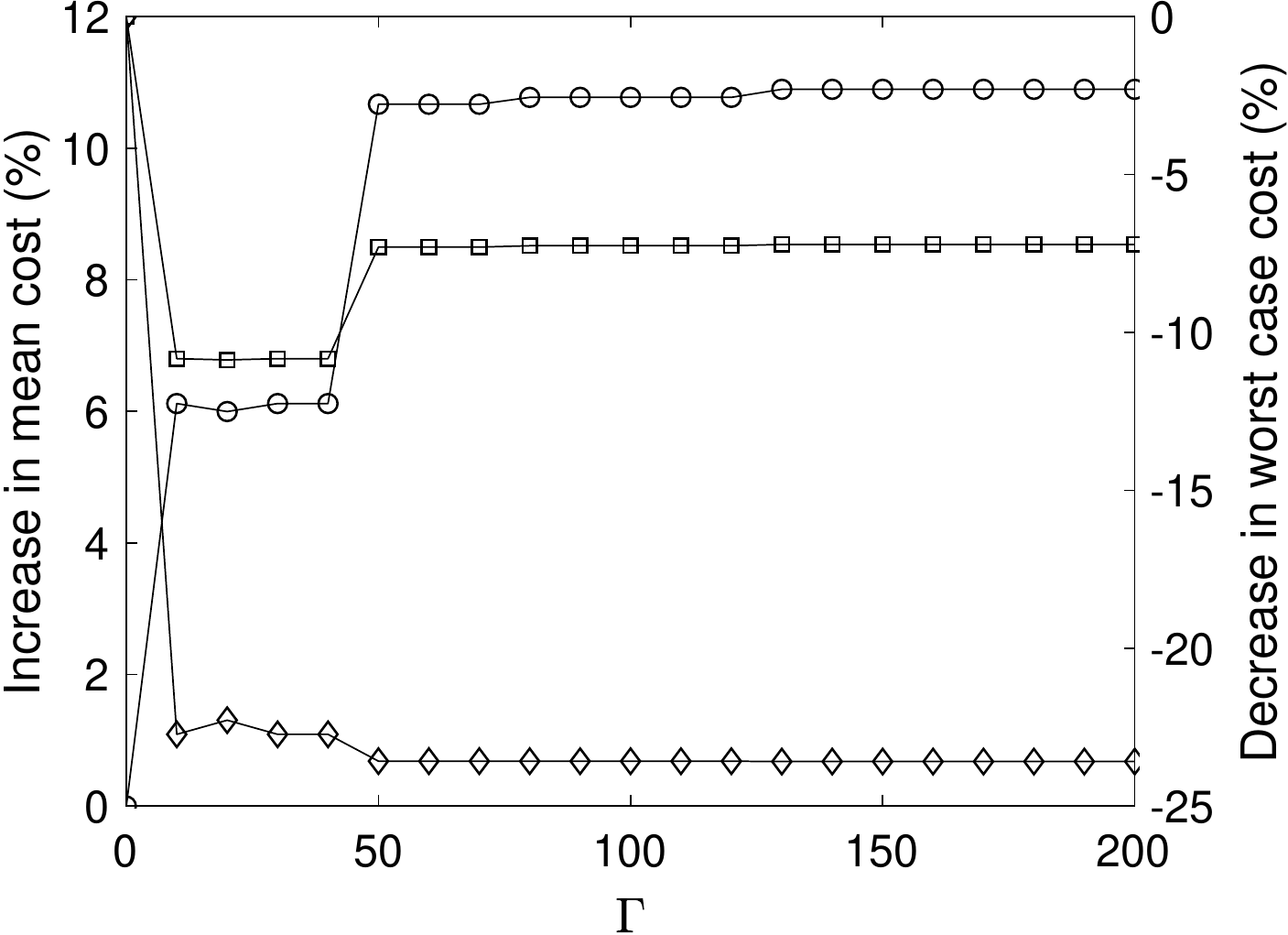}}
    \subfloat[Off-peak]{\includegraphics[width=0.25\linewidth]{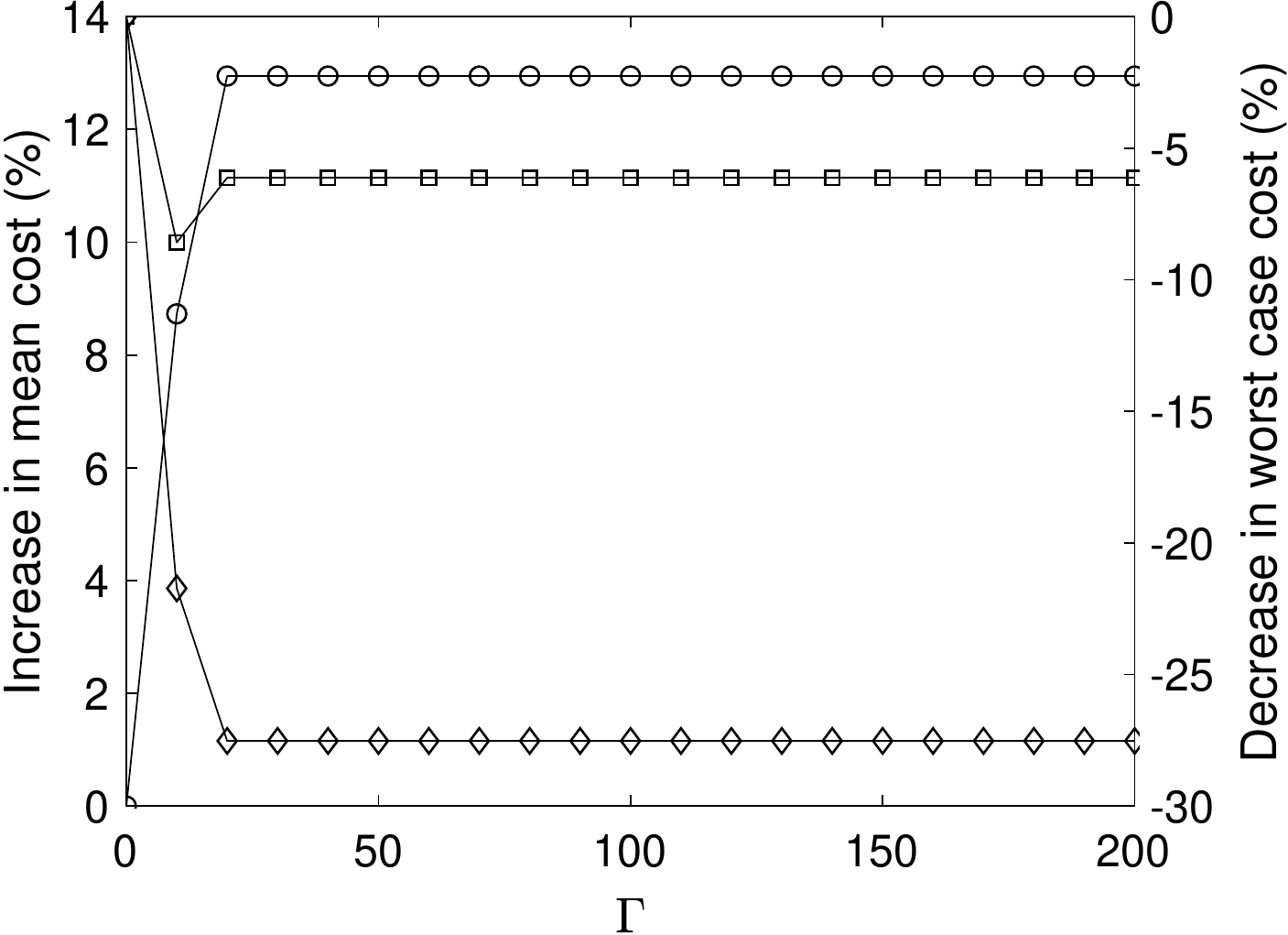}}
    \subfloat[Night time]{\includegraphics[width=0.25\linewidth]{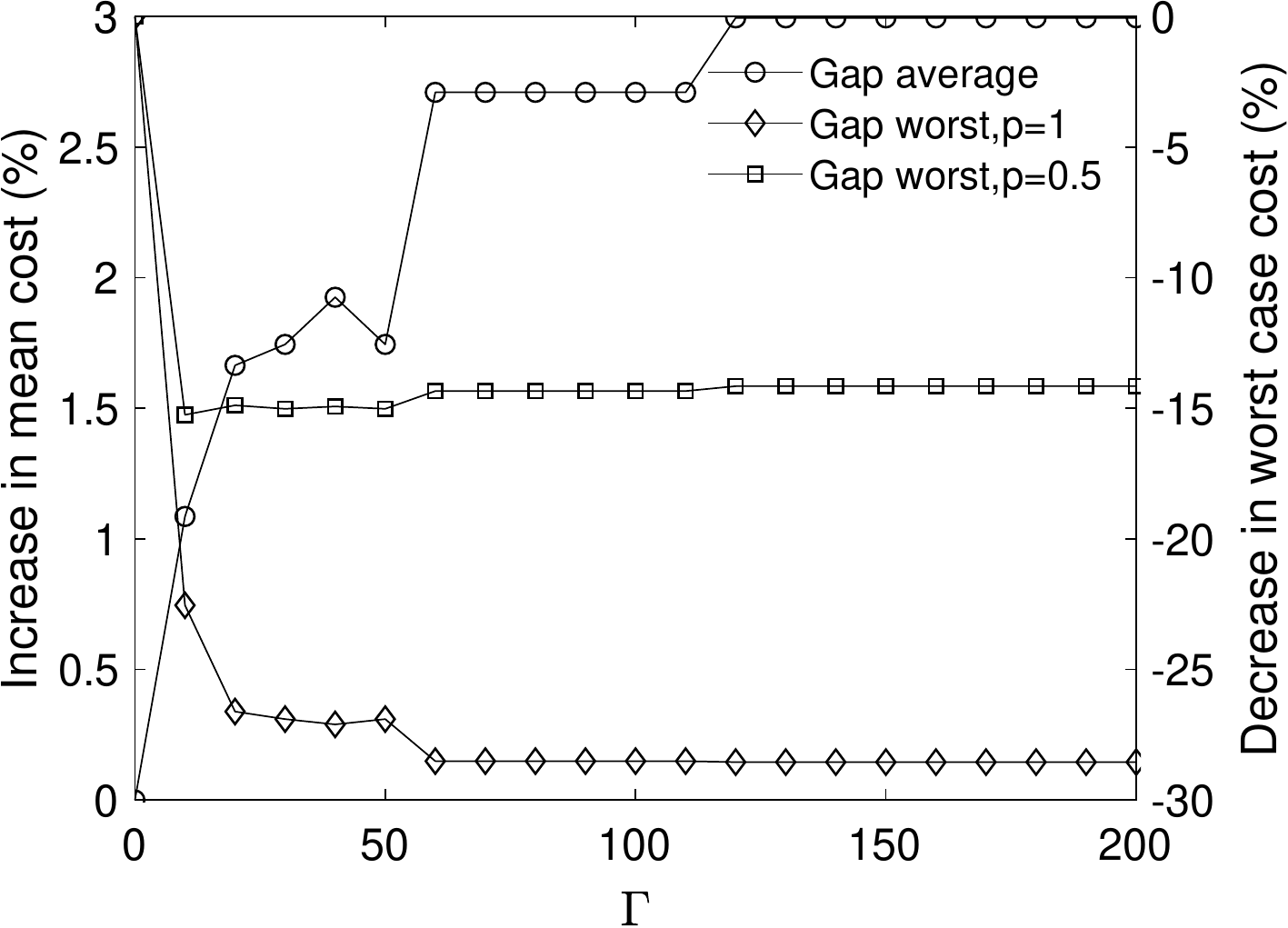}}
    \caption{Change of objective function with respect to different demand uncertainty}
    \label{fig:gap_cost}
\end{figure}

\begin{table}[h!]
\caption{Change of loss of service and operated number of vehicles under the different budget of uncertainty. (L represents the average case, $L_{0.5}$ and $L_1$ corresponds to the worst cases, and $N$ denotes the number of vehicles)}
\tiny
\label{tab:cost_analysis}
\begin{tabular}{lllllllllllllllll}
\toprule
      & \multicolumn{4}{l}{AM}                                           & \multicolumn{4}{l}{PM}    & \multicolumn{4}{l}{Off-peak} & \multicolumn{4}{l}{Night} \\ 
$\Gamma$ & N   & L    & \multicolumn{1}{l}{$L_{0.5}$} & \multicolumn{1}{l}{$L_{1}$} & N   & L    & $L_{0.5}$ & $L_{1}$    & N    & L     & $L_{0.5}$  & $L_{1}$    & N   & L    & $L_{0.5}$ & $L_{1}$    \\ 
\hline
0     & 146 & 1.69 & 6.51                      & 12.25                   & 158 & 0.55 & 5.8  & 11.49 & 157  & 1.97  & 7.38  & 15.29 & 138 & 2.96 & 8.93 & 16.47 \\
10    & 161 & 0.44 & 0.45                      & 4.76                    & 178 & 0    & 0    & 1.31  & 182  & 1.32  & 1.28  & 4.13  & 159 & 0.4  & 0.41 & 4.68  \\
20    & 158 & 1.4  & 1.47                      & 3.89                    & 177 & 0    & 0    & 1.48  & 198  & 0     & 0     & 0.04  & 161 & 0.4  & 0.41 & 2.8   \\
30    & 169 & 0.44 & 0.45                      & 1.09                    & 178 & 0    & 0    & 1.31  & 198  & 0     & 0     & 0.04  & 164 & 0    & 0    & 2.38  \\
40    & 171 & 0.44 & 0.45                      & 0.8                     & 178 & 0    & 0    & 1.31  & 198  & 0     & 0     & 0.04  & 165 & 0    & 0    & 2.27  \\
50    & 171 & 0.44 & 0.45                      & 0.8                     & 189 & 0    & 0    & 0     & 198  & 0     & 0     & 0.04  & 164 & 0    & 0    & 2.38  \\
60    & 172 & 0.44 & 0.45                      & 0.74                    & 189 & 0    & 0    & 0     & 198  & 0     & 0     & 0.04  & 167 & 0    & 0    & 1.48  \\
70    & 171 & 0.44 & 0.45                      & 0.8                     & 189 & 0    & 0    & 0     & 198  & 0     & 0     & 0.04  & 167 & 0    & 0    & 1.48  \\
80    & 172 & 0.44 & 0.45                      & 0.74                    & 190 & 0    & 0    & 0     & 198  & 0     & 0     & 0.04  & 167 & 0    & 0    & 1.48  \\
90    & 172 & 0.44 & 0.45                      & 0.74                    & 190 & 0    & 0    & 0     & 198  & 0     & 0     & 0.04  & 167 & 0    & 0    & 1.48  \\
100   & 172 & 0.44 & 0.45                      & 0.74                    & 190 & 0    & 0    & 0     & 198  & 0     & 0     & 0.04  & 167 & 0    & 0    & 1.48  \\
110   & 172 & 0.44 & 0.45                      & 0.74                    & 190 & 0    & 0    & 0     & 198  & 0     & 0     & 0.04  & 167 & 0    & 0    & 1.48  \\
120   & 172 & 0.44 & 0.45                      & 0.74                    & 190 & 0    & 0    & 0     & 198  & 0     & 0     & 0.04  & 168 & 0    & 0    & 1.41  \\
130   & 172 & 0.44 & 0.45                      & 0.74                    & 191 & 0    & 0    & 0     & 198  & 0     & 0     & 0.04  & 168 & 0    & 0    & 1.41  \\
140   & 172 & 0.44 & 0.45                      & 0.74                    & 191 & 0    & 0    & 0     & 198  & 0     & 0     & 0.04  & 168 & 0    & 0    & 1.41  \\
150   & 172 & 0.44 & 0.45                      & 0.74                    & 191 & 0    & 0    & 0     & 198  & 0     & 0     & 0.04  & 168 & 0    & 0    & 1.41  \\
160   & 172 & 0.44 & 0.45                      & 0.74                    & 191 & 0    & 0    & 0     & 198  & 0     & 0     & 0.04  & 168 & 0    & 0    & 1.41  \\
170   & 172 & 0.44 & 0.45                      & 0.74                    & 191 & 0    & 0    & 0     & 198  & 0     & 0     & 0.04  & 168 & 0    & 0    & 1.41  \\
180   & 172 & 0.44 & 0.45                      & 0.74                    & 191 & 0    & 0    & 0     & 198  & 0     & 0     & 0.04  & 168 & 0    & 0    & 1.41  \\
190   & 172 & 0.44 & 0.45                      & 0.74                    & 191 & 0    & 0    & 0     & 198  & 0     & 0     & 0.04  & 168 & 0    & 0    & 1.41  \\
200   & 172 & 0.44 & 0.45                      & 0.74                    & 191 & 0    & 0    & 0     & 198  & 0     & 0     & 0.04  & 168 & 0    & 0    & 1.41 \\
\bottomrule
\end{tabular}
\end{table}

As we increases the budget of uncertain from 0 to 10, we observe from Table~\ref{tab:cost_analysis} that there is at least a 10\% increase in the number of operated vehicles in all cases. And there may be 41 more vehicles in operation for $\Gamma=200$ than that of $\Gamma=0$ for the off-peak scenarios and the number is at least 26 more as in the AM scenario. Nevertheless, these significantly more number of vehicles in operation may only lead to an increase in average objective function value by 1-2\% for the night time period or by up to 12.95\% during the PM peak period. As a consequence, a majority of the increases in the operation cost are translated into the savings of loss of services and waiting time cost even during the normal demand scenarios. Specifically, we can verify from Table~\ref{tab:cost_analysis} that the additional operation costs contribute to the reduction of 0.55\%~2.96\% of unsatisfied passengers for normal conditions. And the benefits of the robust route scheduling is more evident if we look into the gains during the worst-case scenarios with both $p=0.5$ and $p=1$. In particular, Figure~\ref{fig:gap_cost}(d) suggests that a 2.7\% increase in daily operation costs (with $ \Gamma=60$) may reduce the cost during half demand variations by 14.3\% and by 28.5\% with full demand variations. More importantly, this will also translate into 1.48\% of unsatisfied passengers instead of 16.47\% during the full demand variation cases. Therefore, these findings serve as strong support for the adoption of the robust scheduling for the planning of HHMoD services, especially considering its potential to avoid drastic loss of passengers, which is crucial to the long-term sustainable operation of HHMoD. Even for the most conservative case ($\Gamma>=20$ for off-peak), we still observe a 6.1\% reduction in objective function value with half demand variations and a 27.53\% cost saving with the full demand variations during the off-peak hours. And this again can contribute to reducing the number of unsatisfied passengers from 15.29\% to only 0.04\% during the worst-case scenarios. One may also choose a less conservative strategy at $\Gamma=10$ where significant savings in worst-case scenarios can still be achieved. Consequently, we report that robust route scheduling does not lead to overly conservative solutions during normal operations. At the same time, we can achieve significant savings in terms of operation cost during worst-case scenarios and reductions in the rate of unserved passengers in all scenarios. And these properties make it an ideal tool to support the scheduling of the HHMoD services over a short-time period by offering robust operation strategies against the demand uncertainties. As a final remark, the staircase patterns for the metrics in Figure~\ref{fig:gap_cost} indicates that the performances of the robust route scheduling plateau after certain budget levels, which implies that the further increase in $\Gamma$ will not lead additional cost even during normal operations. In this regard, one may directly solve the single-stage problem with box uncertainty constraints ($\Gamma=200$) instead of the case with a medium level budget of uncertainty. This will obtain the same quality solution with minimum computational time, which could be a potential option if the robust route scheduling needs to be carried out in real-time.

\section{Conclusion}
In this study, we present the three-stage framework for optimal planning and scheduling of the HHMoD service at urban transportation hubs. The proposed framework consists of the route generation, the route combination, and the robust route schedule optimization. We develop efficient and effective route generation algorithms with connectivity to existing public transportation systems to generate the high-quality candidate route set. And a two-stage robust optimization model is developed to find the vehicle and route headway configurations that are resilient to the worst-case realization of passenger demand distributions. We conduct comprehensive numerical experiments for planning the HHMoD at the JFK airport in NYC. To best mimic the real-world settings, we use NYC taxi and FHV data to calibrate the potential passenger demand and use GoogleMap API to obtain the corresponding road traffic information. The results highlight the effectiveness of the proposed route generation algorithms for the HHMoD service, with 5 top candidates route being able to cover nearly 70\% of total passenger demand and fewer than 30 routes are needed to meet all demand across the 100 stops. The results also demonstrate the cost-effectiveness of the robust route schedules during the average demand conditions and its superior performances in reducing the unsatisfied passengers during both normal and worst-case conditions. Moreover, we find that the quality of the generated routes and the performances of the robust route schedules are consistent across different times of the day. 

There are several future directions to extend the scope of our study further. First, as our case study concerns the HHMoD services at the JFK airport, it will be interesting to examine the effectiveness of the solution framework at other activity hubs located in the central urban areas and test the performances in a different city. Second, while the proposed algorithms only consider the given travel time information, the exact and heuristic algorithms may be modified to incorporate the structural properties of the underlying service network (e.g. grid network or the planar graph) and design specific algorithms that can obtain optimal solutions more efficiently. Specifically, as the longest path problem is a special case of our problem, polynomial algorithms may exist on certain types of networks~\cite{ioannidou2009longest}. Finally, the robust scheduling problem can be further extended to account for the fleet with heterogeneous vehicles. This will contribute to the more efficient use of available resources and provide even less conservative operation strategies.


\bibliographystyle{unsrt}
\bibliography{sample}
\section*{Appendix}
In the following, we summarize the implementation details for the shortest-path based heuristic approach. 
\subsubsection*{Initial route generation}
The initial candidate routes are created by finding $k$-shortest routes from each stop $i$ to hubs ~\cite{yen1971finding}. 
The process results in the candidate path set $R^h$ with $k|N|$ candidate routes and we denote the stops that each $r\in R^h$ traverses through as $N^r$. Since there may be multiple candidate routes that pass stop $i$, we then perform passenger demand assignment based on travel time:
\begin{equation}
    P_{i}^r = \frac{e^{-t_{i}^r}}{\sum_{r'\in R_i^h}e^{-t_{i}^{r'}}},  \forall\,i \in N, r \in R_{i}^{h},
\end{equation}\\
where $t_{i}^{r}$ is the route travel time if passengers at stop $i$ take route $r$ and $R_i^h$ is the set of heuristically generated routes that traverse from stop $i$. We can therefore measure the relative importance of each stop and each route as stop weight $W_{i}$ and route weight $W^{r}$ ($\mathbbm{1}_i^r$ denotes whether route $r$ travels by stop $i$):
\begin{equation}
    W_i=\sum_{j\in N} \sum_{r\in R_j^h} \mathbbm{1}_i^r q_jP_{j}^r,\quad  W^r =\sum_{i\in N^r }W_i
\end{equation}
\subsubsection*{Candidate route expansion}
The generated candidate routes in $R^h$ may be further combined to create new routes that extend the coverage of passenger demand. We consider expanding two routes $r_1,r_2\in R^h$ following two rules: 1) if $r_1$ is a sub-route of $r_2$ where $N_{r_1}^h\subseteq N_{r_2}^h$ and 2) if $r_1$ and $r_2$ share common segments. The route expansion is then performed based on crossover operation over the actual link segments of the two routes following Algorithm~\ref{alg:crossover}. And the weight of the expanded route is also calculated by the stop weight $W_{i}$. 
\begin{algorithm}[H]
\begin{algorithmic}[1]
\Require
$r_{1}=\{l_{1},l_{2},...,l_{L}\} $(link segments), 
$r_{2}=\{p_{1},p_{2},...,p_{P}\}$ (link segments),     
$N_{r_1}^h, N_{r_2}^h, G\leftarrow \emptyset$.
\For{$l_{i}=p_{j}$ ($l_i\in r_{1}$, $p_j\in r_{2}$)}
    \If{$N_{r_1}^h \subseteq N_{r_2}^h$}
    \State $\bar{r}\leftarrow \{l_1,l_2,...l_i,p_{j+1},...,p_{P}\}, G\leftarrow G\cup {\bar{r}}$
    \Else
    \State $\bar{r}_{1}$=$\{l_1,l_2,...,l_i,p_{j+1},...,p_{P}\}$, $\bar{r}_{2}$=$\{p_1,p_2,...,p_j,l_{i+1},...,l_{L}\}$, $G \leftarrow G \bigcup \{\bar{r}_{1},\bar{r}_{2}\}$
    \EndIf
\EndFor
\State \Return {$G$} 
\end{algorithmic}
\caption{Candidate Routes Extension}
\label{alg:crossover}
\end{algorithm}

\subsubsection*{Route pruning}
The set of candidate routes may grow rapidly after the route expansion. 
To eliminate unnecessary routes, we conduct route pruning following two metrics as in~\cite{pinelli2016data}. The first metric is route circuity (total route length divided by the euclidean distance). 
We choose routes with length longer than $l_{thd}$, and circuity smaller than $C_{thd}$. The second metric is the subset relationship. If a route is a sub-route of the other route with a smaller length, then the shorter route will be discarded. 
The last metric is route similarity. It quantifies the spatial adjacency of two routes:
\begin{equation}
    S(r_1,r_2) = \frac{\sum_{l\in r_1} \min_{p\in r_2}dist(l,p)+\sum_{p\in r_2} \min_{l\in r_1}dist(p,l)}{L_{r_1}+L_{r_2}}
\end{equation}
where $L_r$ is the length of route $r$ and $dist(p,l)$ is the euclidean distance between the center of two road segments $p$ and $l$. To prune similar routes, we first sort the weight of the routes in $R^h$ in descending order:
\begin{equation}
    W_{i=1}\geq W_{i=2} \geq W_{i=3} \geq ... \geq W_{i=|R^{h}|}
\end{equation}
We then remove the following set of routes from $R^h$ based on route similarity threshold $S_{thd}$:
\begin{equation}
    \{r_{i}| \exists 1 \leq j \leq i-1, S(r_{i},r_{j}) < S_{thd}\}
\end{equation}
In this manner, it removes the repetitive routes with lower weight (importance). 







\end{document}